\documentclass[a4paper]{amsart}
\usepackage[utf8]{inputenc}
\usepackage{amssymb}
\usepackage{amsmath}
\usepackage{mathtools}
\usepackage{mathrsfs}
\usepackage{eufrak}
\usepackage{amsthm}
\usepackage{amsfonts}
\usepackage{textcomp}
\usepackage{graphicx}
\usepackage[pdftex]{color}
\usepackage{paralist}
\usepackage[shortlabels]{enumitem}
\usepackage{hyperref}
\usepackage{comment}
\usepackage[arrow, matrix, curve]{xy}

\newtheorem*{corollary*}{Corollary}
\newtheorem{theorem}{Theorem}[section]

\newtheorem{corollary}[theorem]{Corollary}
\newtheorem{lemma}[theorem]{Lemma}
\newtheorem{proposition}[theorem]{Proposition}

\newtheorem*{claim*}{Claim}

\theoremstyle{definition}
\newtheorem{definition}[theorem]{Definition}
\newtheorem{remark}[theorem]{Remark}
\newtheorem{example}[theorem]{Example}

\newtheorem{observation}[theorem]{Observation}

\theoremstyle{remark}

\numberwithin{equation}{theorem}

\makeatletter
\renewcommand*\env@matrix[1][\
arraystretch]{%
  \edef\arraystretch{#1}%
  \hskip -\arraycolsep
  \let\@ifnextchar\new@ifnextchar
  \array{*\c@MaxMatrixCols c}}
\makeatother

\renewcommand{\mod}{\operatorname{mod}}
\newcommand{\proj}{\operatorname{proj}}
\newcommand{\inj}{\operatorname{inj}}
\newcommand{\Ext}{\operatorname{Ext}}
\newcommand{\End}{\operatorname{End}}
\newcommand{\Hom}{\operatorname{Hom}}
\newcommand{\add}{\operatorname{\mathrm{add}}}
\newcommand{\rad}{\operatorname{\mathrm{rad}}}

\newcommand{\Z}{\mathbb{Z}}
\newcommand{\D}{\mathcal{D}}
\newcommand{\RHom}{\operatorname{\mathsf{R}Hom}}
\newcommand{\Lotimes}{\otimes^\mathsf{L}}
\newcommand{\Ocal}{\mathcal{O}}
\newcommand{\X}{\mathbb{X}}
\newcommand{\Ll}{\mathbb{L}}

\renewcommand{\ker}{\mathrm{Ker}}
\newcommand{\Db}{\mathrm{D}^{\mathrm{b}}}
\newcommand{\Simp}{\mathcal{S}}

\newcommand{\domdim}{\operatorname{domdim}}
\newcommand{\codomdim}{\operatorname{codomdim}}
\newcommand{\idim}{\operatorname{idim}}
\newcommand{\pdim}{\operatorname{pdim}}
\newcommand{\gldim}{\operatorname{gldim}}
\newcommand{\DynA}{\mathbb{A}}
\newcommand{\DynD}{\mathbb{D}}
\newcommand{\DynE}{\mathbb{E}}

\begin{document}

\title{Auslander-Gorenstein algebras from Serre-formal algebras via replication}
\date{\today}

\subjclass[2010]{Primary 16G10,16E10}

\keywords{Representation theory of finite dimensional algebras, Gorenstein dimension, dominant dimension}

\author{Aaron Chan}
\address{Graduate School of Mathematics, Nagoya University, Furocho, Chikusaku, Nagoya 464-8602, Japan}
\email{aaron.kychan@gmail.com}

\author{Osamu Iyama}
\address{Graduate School of Mathematics, Nagoya University, Furocho, Chikusaku, Nagoya 464-8602, Japan}
\email{iyama@math.nagoya-u.ac.jp}

\author{Ren\'{e} Marczinzik}
\address{Institute of algebra and number theory, University of Stuttgart, Pfaffenwaldring 57, 70569 Stuttgart, Germany}
\email{marczire@mathematik.uni-stuttgart.de}

\begin{abstract}
We introduce a new family of algebras, called Serre-formal algebras.
They are Iwanaga-Gorenstein algebras for which applying any power of the Serre functor on any indecomposable projective module, the result remains a stalk complex.
Typical examples are given by (higher) hereditary algebras and self-injective algebras; it turns out that other interesting algebras such as (higher) canonical algebras are also Serre-formal.
Starting from a Serre-formal algebra, we consider a series of algebras - called the replicated algebras - given by certain subquotients of its repetitive algebra.
We calculate the self-injective dimension and dominant dimension of all such replicated algebras and determine which of them are minimal Auslander-Gorenstein, i.e. when the two dimensions are finite and equal to each other.
In particular, we show that there exist infinitely many minimal Auslander-Gorenstien algebras in such a series if, and only if, the Serre-formal algebra is twisted fractionally Calabi-Yau.
We apply these results to a construction of algebras from Yamagata \cite{Yam}, called SGC extensions, given by iteratively taking the endomorphism ring of the smallest generator-cogenerator.
We give a sufficient condition so that the SGC extensions and replicated algebras coincide.
Consequently, in such a case, we obtain explicit formulae for the self-injective dimension and dominant dimension of the SGC extension algebras.
\end{abstract}

\maketitle

\section{Introduction}
The aim of this paper is to give a systematic construction of a distinguished class of Iwanaga-Gorenstein algebras, called minimal Auslander-Gorenstein algebras, which are defined by the  property 
\[\idim A=d+1=\domdim A\]
for some positive integer $d$, where $\idim A$ and $\domdim A$ are the self-injective dimension and dominant dimension of $A$ respectively.
This class includes higher Auslander algebras which are defined by replacing $\idim A$ with the global dimension $\gldim A$.
These algebras appear naturally in the higher Auslander-Reiten theory on $d$-precluster-tilting subcategories \cite{Iya2,IS}.

Our starting point is a famous article of Kunio Yamagata \cite{Yam} on Frobenius algebras, where the following construction was introduced:
For a given finite dimensional algebra $A=A^{[0]}$, we define a series of finite dimensional algebras by
\[A^{[m+1]}:=\End_{A^{[m]}}(A^{[m]}\oplus DA^{[m]}).\]
We call $A^{[m]}$ the \emph{$m$-th SGC extension} of $A$.
At first glance, one would expect that the algebras $A^{[m]}$ become complicated very quickly as $m$ increases.
Nevertheless, Yamagata observed that in the case when $A$ is hereditary, every $A^{[m]}$ has finite global dimension and large dominant dimension - although he did not calculate their explicit values.
One of our main results gives explicit formulae for these dimensions in the case of $A$ being hereditary.

\begin{theorem}\label{first main}
Suppose $A$ is a hereditary algebra.
\begin{enumerate}
\item If $A$ is representation-infinite, then
$\gldim A^{[m]}=2m+1$ and $\domdim A^{[m]}=2m$.
\item If $A$ is representation-finite, then $\gldim A^{[m]}=\domdim A^{[m]}$ holds for infinitely many $m$.
\end{enumerate}
\end{theorem}

The majority of Theorem \ref{first main} is proved by considering the same problem on the \emph{replicated algebras} of $A$.
The $m$-th replicated algebra of $A$ is defined as a $(m+1)\times (m+1)$ matrix algebra:
\[A^{(m)}=\left(\begin{array}{ccccc}
A & 0 & 0 & \cdots & 0  \\
DA & A & 0 &  & 0 \\
\vdots & \ddots & \ddots & \ddots & \vdots \\
0 & & \ddots & \ddots  &\vdots \\
0 & 0 & \cdots & DA & A
\end{array}\right).
\]
Such construction of algebras was studied in, for example, \cite{AssIwa,ABST,LvZh}; but it seems to us that its connection with SGC extensions (Theorem \ref{thm-SGC-replicate}) has not been exploited before; in fact, these two constructions of new algebras from $A$ coincide under a mild assumption.

\begin{theorem}\label{comparison}
Let $A$ be a finite dimensional algebra such that $\Hom_A(DA,A)=0$.
Then the $m$-th SGC extension $A^{[m]}$ is Morita equivalent to the $m$-th replicated algebra $A^{(m)}$.
\end{theorem}

It is usually easier to understand the behaviour of $A^{(m)}$ compared to that of $A^{[m]}$.
As such, the majority of the calculations carried out in this paper are on the replicated algebras.
In fact, we can find the formulae for the dominant dimension and the injective dimension of each indecomposable projective $A^{(m)}$-module in terms of certain combinatorial data (including the Coxeter numbers; see Corollaries \ref{cor-RF} and \ref{cor-RI}).
Theorem \ref{comparison} enables us to deduce Theorem \ref{first main} from these calculations for $A^{(m)}$, apart from one particular case (namely, when $A$ is a Nakayama algebra).

In fact, our main result is much more general, and is stated for a much wider class of algebras - we call them \emph{Serre-formal algebras}.
Recall that the perfect derived category of an Iwanaga-Gorenstein algebra $A$ has a Serre functor $\nu$.
We call an Iwanaga-Gorenstein algebra $A$ \emph{Serre-formal} if $\nu^i(A)$ is isomorphic to a direct sum of stalk complexes for any $i\in\mathbb{Z}$ in the bounded derived category $\Db(\mod A)$ of finitely generated $A$-modules.
This class of algebras is closed under tensor products (Proposition \ref{prop-tensor}), and contains several important subclasses which are central in representation theory. 
For example, self-injective algebras and higher hereditary algebras - such as path algebras of finite acyclic quivers and Beilinson algebras - are all Serre-formal.
Moreover, algebras arising from symmetrisable Cartan matrices \cite{GLS}, and Ringel's canonical algebras \cite{Rin} as well as higher canonical algebras of \cite{HIMO} (Theorem \ref{thm-canonical}), are also Serre-formal.

From the perspective of this paper, the main advantage of considering Serre-formal algebras is that we can determine the explicit formulae for the dominant and injective dimensions of each indecomposable projective $A^{(m)}$-module (Theorem \ref{thm-dims2}).
Hence, we can obtain a replicated algebra version of Theorem \ref{first main} for Serre-formal algebras.

We call an algebra $A$ \emph{periodically Serre-formal} if for each indecomposable projective $A$-module $P$, there exists a positive integer $\ell$ so that $\nu^{-\ell}(P)$ is also projective.
It is easy to see that in the case of $A$ being hereditary, this is equivalent to $A$ being representation-finite.
Moreover, recall from \cite{HerIya} that an algebra $A$ is \emph{twisted fractionally Calabi-Yau} if there is a pair of integers $(m\neq 0,n)$ so that $\nu^m(A)\cong A[n]$ in $\Db(\mod A)$.
In such a case, we say that the twisted Calabi-Yau dimension of $A$ is $(m,n)$ if $m$ is the smallest positive integer for which such an isomorphism holds.
We will show the following equivalences (see Proposition \ref{prop-twistedCY}, Corollary \ref{cor-minAG2}).

\begin{theorem}\label{second main}
Let $A$ be a finite dimensional algebra which is Serre-formal.
Then the following conditions are equivalent.
\begin{itemize}
\item $A$ is periodically Serre-formal.
\item $A$ is twisted fractionally Calabi-Yau.
\item $A^{(m)}$ is minimal Auslander-Gorenstein for some (or infinitely many) positive integer $m$.
\end{itemize}
In such a case, let $(h,s)$ be the twisted Calabi-Yau dimension of $A$, then $A^{(m)}$ being minimal Auslander-Gorenstein is equivalent to $m=ht-1$ for some integer $t>0$.
Moreover, we have
\[
\idim A^{(m)} = (h+s)t-1 =\domdim A^{(m)}.
\]
\end{theorem}

In the special case of higher hereditary algebras, we have the following result (see Corollary \ref{cor-dim-dhered}, Proposition \ref{prop-RF-d}).
We refer the reader to Subsection \ref{subsec-NF-Bkgd} and Subsection \ref{subsec-NF-d-hered} for details on the terminologies.
\begin{theorem}\label{first main-d}
Let $A$ be a $d$-hereditary algebra with $d\ge1$.
\begin{enumerate}
\item If $A$ is $d$-representation-infinite, then
$\gldim A^{(m)}=(d+1)m+d$ and $\domdim A^{(m)}=(d+1)m$.
\item If $A$ is $d$-representation-finite, then $\gldim A^{(m)}=\domdim A^{(m)}$ holds for infinitely many $m$.
\end{enumerate}
\end{theorem}

The structure of the paper is as follows.
In Section \ref{sec-Repl-Alg}, we present some elementary properties of $m$-th replicated algebras $A^{(m)}$ of arbitrary finite dimensional algebra $A$.
In Section \ref{sec-Hered-base}, we give more in-depth studies - in particular, the homological dimensions - of $A^{(m)}$ in the case when $A$ is hereditary.
In Section \ref{sec-Nu-formal}, we introduce Serre-formal algebras, study their basic properties (Subsection \ref{subsec-NF-Bkgd}), and give basic examples (Subsections \ref{subsec-NF-naka}, \ref{subsec-NF-canonical}).
Moreover, we investigate $A^{(m)}$ for the case when $A$ is Serre-formal in a similar manner as Section \ref{sec-Hered-base}, which yields Theorem \ref{second main} (Subsection \ref{subsec-NF-dims}) and Theorem \ref{first main-d} (Subsection \ref{subsec-NF-d-hered}).
In Section \ref{sec-SGC}, we explain the connection between the $m$-th replicated algebra $A^{(m)}$ and the $m$-th SGC extension $A^{[m]}$ (Subsection \ref{subsec-SGC-replicate}); hence, showing Theorem \ref{comparison}.
We then complete the proof of Theorem \ref{first main} by calculating the dominant and global dimension of certain Nakayama algebras explicitly (Subsection \ref{subsec-SGC-typeA}).

\section*{Acknowledgement}
This project started with numerous fruitful experiments thanks to \cite{QPA}.
AC is supported by JSPS International Research Fellowship, and partly by IAR Research Project, Institute for Advanced Research, Nagoya University.

\section{Preliminaries}\label{sec-Prelim}
Fix throughout a field $K$.
Algebras will always mean finite dimensional $K$-algebras.
We follow the notation and conventions of \cite{SkoYam}.
In particular, all modules are finite dimensional right modules and all maps are module maps unless otherwise stated.
The $K$-linear duality is denoted by $D(-)$.
Unless otherwise specified, all algebras are ring-indecomposable, non-simple, finite dimensional over $K$, and basic.

Let $A$ be an algebra.
The Nakayama functor $D\Hom_A(-,A_A)$ and its inverse are denoted by $\nu_A$ and $\nu_A^{-}$ respectively.
The Auslander-Reiten translate and its inverse are denoted by $\tau_A$ and $\tau_A^{-}$ respectively.
The full subcategory of projective (resp. injective) $A$-modules is denoted by $\proj A$  (resp. $\inj A$).
For a module $M$, we denote by $|M|$ the number of indecomposable direct summands of $M$ up to isomorphism.
We will usually denote an idempotent by $e$, or by $e_i$ for a label $i$ of (an isomorphism class of) a simple module.

\medskip

Let $M,N$ be $A$-modules.
We say that $M$ has an ($\add N$)-resolution (resp. ($\add N$)-coresolution) if there is an infinite exact sequence
\begin{align}
\begin{array}{rlllll}
\cdots & \to N_2 & \to N_1 & \to N_0 & \to M & \to 0 \\
(\text{resp. } \quad 0 &\to M & \to N^0& \to N^1& \to N^2 &\to \cdots)
\end{array}\notag 
\end{align}
so that $N_i\in \add N$ (resp. $N^i\in\add N$).

Let $M$ be an $A$-module with minimal injective coresolution $I^\bullet=(I^n)_{n\geq 0}$ and minimal projective resolution $P_\bullet=(P_n)_{n\geq 0}$.
The \emph{dominant dimension} and \emph{codominant dimension} of $M$ are
\begin{align}
\domdim M& :=\begin{cases}
1+\sup \{ n \geq 0 \mid I^j \in \proj A \text{ for } j=0,1,...,n \} & \mbox{ if $I^0$ is projective,}\\
0 & \mbox{ if $I^0$ is not projective;}
\end{cases}\notag \\
\codomdim M & :=\begin{cases}
1+\sup \{ n \geq 0 \mid P_j \in \inj A \text{ for } j=0,1,...,n \} & \mbox{ if $P_0$ is injective,}\\
0 & \mbox{ if $P_0$ is not injective.}
\end{cases}\notag
\end{align}
respectively.

The \emph{dominant dimension} of $A$, denoted by $\domdim A$, is the dominant dimension of the regular $A$-module $A_A$.
Note that we always have $\domdim A = \codomdim DA$, see \cite{Ta} for example.

\medskip

The \emph{left} (resp. \emph{right}) \emph{self-injective dimension of $A$} is the injective dimension of $A$ as a left (resp. right) $A$-module.
An algebra $A$ is called \emph{Iwanaga-Gorenstein} if both its left and right self-injective dimensions are finite.
When these dimensions are finite, then they are equal.
If this number is equal to $d<\infty$, then we simply call $A$ a $d$-Iwanaga-Gorenstein algebra \cite{EncJen}.
Note that this is \emph{not} the same as the $k$-Gorenstein algebras which appear in \cite{AR}.

For convenience, we recall the following.
\begin{lemma}\label{lem-IGfacts}
Let $A$ be a $d$-Iwanaga-Gorenstein algebra, and $M$ be an $A$-module.
\begin{enumerate}[(1)]
\item {\rm (Iwanaga, 1980)} The following are equivalent.
\begin{itemize}
\item $\pdim_AM\leq d$.
\item $\idim_AM\leq d$.
\item $\pdim_AM< \infty$.
\item $\idim_AM< \infty$.
\end{itemize}

\item The following are equivalent for non-projective $M$.
\begin{itemize}
\item $M$ is the zero-th homology of a totally acyclic complex in $\mod A$ (i.e. it is Gorenstein projective).
\item $M\cong \Omega_A^d(X)$ for some $X\in\mod A$.
\item $\Ext_A^i(M,A)=0$ for all $i\in\{1,2,\ldots, d\}$ (i.e. $M$ is maximal Cohen-Macaulay).
\end{itemize}
\end{enumerate}
\end{lemma}

Every algebra of finite global dimension $d$ is $d$-Iwanaga-Gorenstein.
On the other hand, a $d$-Iwanaga-Gorenstein algebra has global dimension either $d$ or $\infty$.

\begin{definition}
A non-self-injective algebra $A$ is \emph{higher Auslander} (resp. \emph{minimal Auslander-Gorenstein}), or more specifically \emph{$d$-Auslander} (resp. \emph{$d$-Auslander-Gorenstein}), if $\gldim A=d+1=\domdim A$ (resp. $\idim {}_AA=\idim A_A=d+1=\domdim A$) for some positive $d$.
\end{definition}

\section{Replicated Algberas}\label{sec-Repl-Alg}
Let us review the construction of replicated algebras and the description of their modules.

Recall that for a finite dimensional algebra $A$, there is an infinite dimensional locally bounded algebra called \emph{repetitive algebra} $\hat{A}$.
The underlying vector space is given by $(\bigoplus_{i\in\Z}A)\oplus(\bigoplus_{i\in\Z}DA)$, with elements of the form $(a_i,f_i)_{i\in\Z}$ with finitely many non-zero $a_i$ and $f_i$'s.
The multiplication is defined by $(a_i,f_i)_i(b_i,g_i)_i=(a_ib_i,a_{i+1}g_i+f_ib_i)_i$.
$\hat{A}$ can be regarded as a $\Z\times \Z$ matrix (non-unital but locally unital) algebra with diagonal entries $\hat{A}_{i,i}=A$, sub-diagonal entries $\hat{A}_{i,i-1}=DA$, and everywhere else zero.

\begin{definition}
Let $A$ be a finite dimensional algebra.
Denote by $e^{(m)}$ the idempotent of $\hat{A}$ given by the $\Z\times\Z$ matrix with $(i,i)$-th entry $1$ for all $i\in\{0,1,\ldots,m\}$ and zero everywhere else.
For $m\geq 0$, the \emph{$m$-replicated algebra} $A^{(m)}$ of $A$ is the idempotent truncation $e^{(m)}\hat{A}e^{(m)}$.
In other words, we have
\begin{align}
A^{(m)}:=e^{(m)} \hat{A}e^{(m)} =
\left(\begin{array}{cccccc}
A & 0 & 0 & \cdots & 0 & 0 \\
DA & A & 0 &  & & 0 \\
0 & DA & A &  & & 0 \\
\vdots & \ddots & \ddots & \ddots & &\vdots \\
0 & & \ddots & \ddots & \ddots &\vdots \\
0 & 0 & 0 & \cdots & DA & A
\end{array}\right), \notag 
\end{align}
where the matrix is of size $(m+1)\times (m+1)$.
We call $A$ the \emph{base algebra}.
\end{definition}

\begin{example}\label{eg-typeA-Kronecker}
(1) Let $A$ be the path algebra of type $\DynA_n$ with linear orientation (i.e. $A$ is a Nakayama algebra with $n$ simple modules).
Then $A^{(m)}=KQ/I$ with $Q$ of type $\DynA_{n(m+1)}$ and $I$ generated by all paths of length $n+1$.

(2) Let $A$ be the Kronecker algebra $K(1\rightrightarrows 2)$.
Then $A^{(m)} = KQ/I$ with
\[
Q: \xymatrix@C=50pt{ 1 \ar@/^/[r]^{a_1} \ar@/_/[r]_{b_1} & 2 \ar@/^/[r]^{a_2} \ar@/_/[r]_{b_2} & 3  \ar@{.}[rr] & & 2m+1\ar@/^/[r]^{a_{2m+1}} \ar@/_/[r]_{b_{2m+1}} & 2m+2 }
\]
and $I$ is generated by $a_ja_{j+1}$, $b_jb_{j+1}$, $a_jb_{j+1}-b_ja_{j+1}$, and all paths of length 3.
\end{example}

\begin{proposition}\label{prop-repli-tensor-SI}
Suppose $A, B$ are finite dimensional algebras with $B$ basic and self-injective.
Then there is an algebra isomorphism $(A\otimes_KB)^{(m)}\cong A^{(m)}\otimes_KB$.
In particular, $(A\otimes_KB)^{(m)}$ is minimal Auslander-Gorenstein if, and only if, so is $A^{(m)}$.
\end{proposition}
\begin{proof}
For simplicity, we take $\otimes:=\otimes_K$.
Suppose ${}_\Gamma M_\Lambda$ is a $\Gamma$-$\Lambda$-bimodule and $\phi$ is an algebra automorphism of $\Lambda$.
Then there is an algebra isomorphism 
\begin{align*}
\left(\begin{array}{cc}\phi^{-1} & 0\\ \mathrm{id}  & \mathrm{id}\end{array}\right):
\left(\begin{array}{cc}\Lambda & 0\\ M & \Gamma\end{array}\right) 
\to \left(\begin{array}{cc}\Lambda & 0 \\ {}_1M_\phi & \Gamma\end{array}\right),
\end{align*}
where ${}_1M_\phi$ denotes the $\Gamma$-$\Lambda$-bimodule whose right $\Lambda$-action is twisted by $\phi$.

Since $B$ is basic and self-injective, there is an algebra automorphism $\phi$ of $B$ so that $DB\cong {}_1B_\phi$ as $B$-bimodule.
In particular, we have $A\otimes B$-bimodules isomorphisms $D(A\otimes B)\cong DA\otimes DB \cong DA\otimes {}_1B_\phi \cong {}_1(DA\otimes B)_{\mathrm{id}\otimes\phi}$.
Now the claim for $m=1$ follows from the discussion of the first paragraph by taking $\Lambda=\Gamma=A\otimes B$.
We omit the similar proof for $m>1$.
\end{proof}

Any $A^{(m)}$-module can be described by a sequence
\[
(M_0, f_1, M_1, f_2, \ldots, f_m, M_m) ,
\]
where $M_i\in \mod A$ and $f_i\in \Hom_A(M_i\otimes_ADA,M_{i-1})$ satisfying $f_{i-1}\circ (f_i\otimes 1)=0$ for all $i$.
We denote such a module simply by $(M_i,f_i)_i$ whenever $m$ is understood.
Let $(a_i,\phi_i)_i$ be an element of $A^{(m)}$, and $(x_i)_i$ be an element of the module given by the above sequence.
Then the $A^{(m)}$-action is given by $(x_i)_i\cdot (a_i,\phi_i)_i = (x_ia_i+ f_{i+1}(x_{i+1}\otimes \phi_{i+1}))_i$.
Using this description, an $A^{(m)}$-module homomorphism can be described by a sequence $(\theta_i)_i: (M_i,f_i)_i \to (N_i,g_i)_i$ with $\theta_i\in \Hom_A(M_i,N_i)$ satisfying the natural commutation relation (c.f. \cite[III.2]{ARS}).

\begin{observation}\label{eg-modules}
Let $M$ be an $A$-module
\begin{enumerate}[(1)]
\item For $0\leq k\leq m$, $[M]_k$ is the ``stalk module" given by a row vector with $m+1$ entries, where the $k$-th entry being $M$ and the remaining zero:
\[
[M]_k := (\underbrace{\stackrel{0}{0}, \ldots, \stackrel{k-1}{0}, \stackrel{k}{M}, \stackrel{k+1}{0}, \ldots, \stackrel{m}{0}}_{m+1\text{ entries}});
\]
the action of $A^{(m)}$ is obvious.
Note that, for every $0\leq k\leq m$, $[?]_k:\mod A\to \mod A^{(m)}$ is a full exact embedding of categories which preserves indecomposable modules, almost split sequences, and irreducible morphisms (c.f. \cite[Lemma 5]{ABST}).

\item If $M$ is projective (resp. injective), for $0\leq k < m$, $[\nu_A(M),M]_k$ (resp. $[M,\nu_A^{-1}(M)]_k$) is the $A^{(m)}$-module given by $M_{k+1}=\nu_A(M)$ (resp. $M$), $f_{k+1}=\mathrm{id}$, $M_{k}=M$ (resp. $\nu_A^-(M)$), i.e.
\begin{align}
[\nu_A(M),M]_k &:= (\underbrace{\stackrel{0}{0}, \ldots, \stackrel{k-1}{0}, \stackrel{k}{\nu_A(M)}, \stackrel{k+1}{M}, \stackrel{k+2}{0},\ldots, \stackrel{m}{0}}_{m+1\text{ entries}}) \quad \text{for projective $M$;}\notag \\
[M,\nu_A^-(M)]_k &:= (\underbrace{\stackrel{0}{0}, \ldots, \stackrel{k-1}{0}, \stackrel{k}{M}, \stackrel{k+1}{\nu_A^-(M)}, \stackrel{k+2}{0}, \ldots, \stackrel{m}{0}}_{m+1\text{ entries}}) \quad \text{for injective $M$.}\notag 
\end{align}
This also shows that $[\nu_A(M),M]_k$ is projective-injective as it is a direct summand of the $(k-1)$-st row of $A^{(m)}$, as well as the dual of a direct summand of the $k$-th column of $A^{(m)}$.

\item The injective envelope $I_M$ and cosyzygy $\Omega_A^-(M)$ form a short exact sequence
\[
0\to M\xrightarrow{\iota} I_M \xrightarrow{\pi} \Omega_A^-(M)\to 0.
\]
For $0\leq k< m$, the injective envelop of $M$ is $[I_M,\nu_A^-(I_M)]_k$.
Moreover, we have a short exact sequence
\begin{align}\label{eq-cosyzygy}
0\to [M]_k\xrightarrow{[\iota,0]_k} [I_M,\nu_A^-(I_M)]_k \xrightarrow{[\pi,\mathrm{id}]_k} \Omega_{A^{(m)}}^-([M]_k)\to 0,
\end{align}
where the cosyzygy $\Omega_{A^{(m)}}^-([M]_k)$ of $[M]_k$ is given by $(M_i,f_i)_i$ with $M_{k+1}=\Omega_A^-(M)$, $f_{k+1}=\pi$, $M_{k}=\nu_A^-(I_M)$, and everything else zero.
In particular, if $M$ is injective, then $\Omega_{A^{(m)}}^-([I]_k)=[\nu_A^-(I)]_{k+1}$.
Note that $[\pi,\mathrm{id}]_k$ is the map given by the sequence $(\theta_i)_{i=0,1,\ldots, m}$ with $\theta_k=\pi, \theta_{k+1}=\mathrm{id}, \theta_i=0$ for all $i\neq k$; similarly for $[\iota,0]_k$.
We leave it to the reader to write down the analogous description of $\Omega_{A^{(m)}}([M]_k)$.
\end{enumerate}
\end{observation}

In addition to (2), we remark that
\begin{align}
\proj A^{(m)} &= \add\{ [A]_0, [DA,A]_i \mid i=0,1,\ldots,m-1 \}; \notag \\
\inj A^{(m)} &= \add\{ [DA]_m, [DA,A]_i \mid i=0,1,\ldots,m-1 \}. \notag 
\end{align}

The following is well-known.
\begin{lemma}{\rm \cite{AssIwa}}\label{lem-gldim}
There is an inequality $m+\gldim A \leq \gldim A^{(m)} \leq (m+1)\gldim A+m.$
\end{lemma}

\section{Hereditary base algebra}\label{sec-Hered-base}
In subsection \ref{subsec-hered-dim}, we calculate the homological dimensions of the replicated algebras of hereditary algebras, and use them to determine when higher Auslander algebras appear.
Then we will give some other simple observations on these replicated algebras in subsection \ref{subsec-hered-further}.

Throughout this section, we assume $A$ is a basic, ring-indecomposable, hereditary, non-simple algebra.
This means that there is a finite acyclic valued quiver $Q$ associated to $A$ which is not an isolated vertex.
For detailed reference, including the classification of representation-finite hereditary algebras, we refer the reader to \cite[Chapter VIII]{ARS}.

Before we begin our exposition, we specifically remind the reader that $\nu_A^-(M)=\Hom_A(DA,M)=0$ for all indecomposable non-injective $A$-modules $M$.
Moreover, there is an automorphism $\nu_1:=\RHom_A(DA,?[1])$ on the bounded derived category $\Db(\mod A)$.
It is well-known that
\begin{align}\label{vanish}
\nu_1^{-k}(M)\in (\mod A)[s_M(k)]
\end{align}
for some $s_M(k)\in \Z_{\geq 0}$, for each indecomposable $A$-module $M$.
We denote by $M^{\succ k}$ the $A$-module $\nu_1^{-k}(M)[-s_M(k)]$.
Note that $\nu_1^{-1}(M)=\tau_A^-(M)$ for any non-injective $A$-module $M$.

\subsection{Homological dimensions}\label{subsec-hered-dim}
By Lemma \ref{lem-gldim} and the assumption that $\gldim A\leq 1$, we have
\[
\gldim  A^{(m)} = \idim A^{(m)} < \infty \;\;\text{ for all }m\geq 1.
\]

\begin{lemma}\label{lem-injcopres}
Let $M$ be a non-injective $A$-module with minimal injective coresolution $0\to I^0 \xrightarrow{f} I^1 \to 0$.
For $0\leq i< m$, we have an exact sequence of $A^{(m)}$-modules:
\begin{align}\label{eq-inj-pres}
0 \to [M]_i \to [I^0,\nu_A^-(I^0)]_i \xrightarrow{[f,\nu_A^-(f)]_i} [I^1,\nu_A^-(I^1)]_i\to [\tau_A^-(M)]_{i+1} \to 0,
\end{align}
where the middle two terms are the minimal injective copresentation of $[M]_i$ in $\mod A^{(m)}$.
\end{lemma}
Note that the middle two terms are in $\add([DA,A]_i)$; hence, they are projective-injective.
\begin{proof}
Using Observation \ref{eg-modules}, we obtain the first two terms of (\ref{eq-inj-pres}).
To obtain the third term, observe that $[f,\nu_A^-(f)]_i$ factors through $\Omega_{A^{(m)}}^-([M]_i)$ as follows:
\[
[I^0,\nu_A^-(I^0)]_i \xrightarrow{[\pi,\mathrm{id}]_i} \Omega_{A^{(m)}}^-([M]_i) \xrightarrow{[\mathrm{id},\nu_A^-(f)]_i} [I^1, \nu_A^-(I^1)]_i,
\]
where $\pi$ is the canonical projection from $I^0$ to $\Omega_A^-(M)$.
As $\ker(\nu_A^-(f))=\nu_A^-(M)=0$, the sequence is exact at $[I^0,\nu_A^-(I)]_i$.
Since $\mathrm{Coker}(f)=0$, we have $\mathrm{Coker}[f,\nu_A^-(f)]_i=[\mathrm{Coker}(\nu_A^-(f))]_{i+1}=[\tau_A^-(M)]_{i+1}$.
\end{proof}

Recall that for any indecomposable $A$-module $M$, we have
\begin{align}
\nu_1^-(M) \cong \begin{cases}
\tau_A^-(M), & \text{if $M$ is non-injective;}\\
\nu_A^-(M)[1], & \text{if $M$ is injective.}
\end{cases}\label{eq-nuD}
\end{align}
Using this and the fact that $\Omega_{A^{(m)}}^-([I]_i)=[\nu_A^-(I)]_{i+1}$ for any injective $A$-module $I$ yields the following list of exact sequences.
Let $M$ be an indecomposable projective $A$-module.
\begin{enumerate}[(i)]
\item If $M^{\succ k}$ is injective and $0\leq i< m$, then we have
\begin{align}  
\xi(M,k,i):\qquad 0 \to [M^{\succ k}]_i \to I \to [M^{\succ k}]_{i+1} \to 0, \label{eq-seq1}
\end{align}
where $I$ is a projective-injective $A^{(m)}$-module.

\item If $M^{\succ k}$ is non-injective and $0\leq i< m$, then we have
\begin{align}
\xi(M,k,i):\qquad 0 \to [M^{\succ k}]_i \to I \to I' \to [M^{\succ k+1}]_{i+1} \to 0,\label{eq-seq2}
\end{align}
where $I,I'$ are projective-injective $A^{(m)}$-modules.

\item If $i=m$, then we have
\begin{align}
\xi(M,k,m):\qquad 0 \to [M^{\succ k}]_m \to [I]_m \to [I']_m \to  0,\label{eq-seq3}
\end{align}
where $I$ is a non-zero injective $A$-module, and so is $I'$ when $M^{\succ k}$ is non-injective; otherwise, $I'=0$.
\end{enumerate}

For any indecomposable $A$-module $M$ and any $0\leq i\leq m$, the minimal injective coresolution of the $A^{(m)}$-module $[M]_i$ can then be obtained by iteratively taking the Yoneda products of 
\[
\xi(M,0,i), \xi(M,1,i+1), \cdots, \xi(M,m-1-i,m-1), \xi(M,m-i,m).
\]

\begin{proposition}\label{prop-dim}
Let $M$ be an indecomposable $A$-module and $i$ be an integer with $0\leq i\leq m$.
Then the following hold.
\begin{enumerate}
\item $\idim_{A^{(m)}}[M]_i-\domdim_{A^{(m)}}[M]_i = \idim_A M^{\succ m-i}$.
\item $\domdim_{A^{(m)}}[M]_i= 2(m-i)-s_M(m-i)$.
\item $\idim_{A^{(m)}}[M]_i = 2(m-i)-s_M(m-i+1)$
\end{enumerate}
\end{proposition}
\begin{proof}
From the preceding discussion, we see that every term in the minimal injective coresolution of $[M]_i$ is projective-injective apart from those given by $\xi(M,m-i,m)$.

Therefore, these values follow immediately from the description (\ref{eq-seq3}) of $\xi(M,m-i,m)$ and the fact that $[I]_m$ is a non-projective $A^{(m)}$-module for any injective non-projective $A$-module $I$.
\end{proof}

From (1) of the proposition, it makes sense to define
\[
\epsilon_M(k):=\begin{cases}
0, & \text{if $M^{\succ k}$ is injective;}\\
1, & \text{otherwise.}
\end{cases}
\]
Then we can write down the dominant and global dimension of $A$ in two simple formulae.

\begin{theorem}\label{thm-dims}
Let $A$ be a non-simple hereditary algebra, then we have
\begin{align}
\domdim A^{(m)} & = 2m-\max\{s_P(m)\mid P\mbox{ an indecomposable projective $A$-module}\} \notag \\
 \text{ and }\quad \gldim A^{(m)} &= 2m-\min\{s_P(m)-\epsilon_P(m)\mid P\mbox{ an indecomposable projective $A$-module}\}\notag 
\end{align}
for all $m\geq 1$.
\end{theorem}
\begin{proof}
Every indecomposable projective non-injective $A^{(m)}$-module is of the form $[P]_0$, where $P$ is an indecomposable projective $A$-module.
By Proposition \ref{prop-dim}, we have
\begin{align*}
\domdim_{A^{(m)}}[P]_0 = 2m-s_M(m) \quad \text{ and }\quad \idim_{A^{(m)}}[P]_0 = 2m-(s_M(m)-\epsilon_P(m)).
\end{align*}
By definition, the dominant and injective dimensions of $A^{(m)}$ are the minimum and maximum of these respectively.
Since $\gldim A^{(m)}$ is finite by Lemma \ref{lem-gldim}, it is equal to $\idim_{A^{(m)}} A^{(m)}$.
\end{proof}
Theorem \ref{thm-dims} implies for $m\geq 1$ that $\domdim A^{(m)} \geq m+1$ when $A$ is non-uniserial or $m\geq 2$.
This is because we have $s_P(m)<m$ in these situations.
On the other hand, the dominant dimension is usually much larger than this bound, which we will explain in the following.

\begin{corollary}\label{cor-RI}
If $A$ is representation-infinite hereditary algebra, then we have, for all $m\geq 1$,
\[
\domdim A^{(m)} = 2m  \quad \text{ and }\quad \gldim A^{(m)}=\idim A^{(m)}= 2m+1.
\]
In particular, $A^{(m)}$ is neither minimal Auslander-Gorenstein nor higher Auslander for all $m\geq 1$.
\end{corollary}
\begin{proof}
When $A$ is representation-infinite, then we have $\nu_1^{-k}(P)\cong \tau_A^{-k}(P)$ (i.e. $s_M(k)=0$) for all $k\geq 0$.
The claim follows from Theorem \ref{thm-dims}.
\end{proof}
Part of this result, namely $\gldim (A^{(m)})=2m+1$, was also shown in \cite[Lem 4.3]{LvZh}.

In the following, we assume $A$ is representation-finite hereditary, and determine the dominant and global dimension of $A^{(m)}$.

We fix the enumeration of vertices of Dynkin graphs which have non-trivial automorphisms as follows:
\begin{align}
\DynA_n: & \xymatrix{ 1 \ar@{-}[r] & 2 \ar@{-}[r] & \cdots \ar@{-}[r]& n-1 \ar@{-}[r] & n} \notag \\
\DynD_n: & \xymatrix@R=10pt{ 1 \ar@{-}[r] & 2 \ar@{-}[r] & \cdots \ar@{-}[r]& n-2\ar@{-}[r]\ar@{-}[d] & n-1 \\  & & &n& } \notag \\
\DynE_6: & \xymatrix@R=10pt{ 1 \ar@{-}[r] & 2 \ar@{-}[r] & 3 \ar@{-}[d]\ar@{-}[r] & 4\ar@{-}[r] & 5 & \\  & & 6&&& } \notag 
\end{align}

For a Dynkin graph $\Delta \in \{\DynA_n, \mathbb{B}_n, \mathbb{C}_n,\DynD_n, \DynE_6, \DynE_7, \DynE_8, \mathbb{F}_4, \mathbb{G}_2\}$, we denote by $h_\Delta$ the Coxeter number of type $\Delta$, which are given in the following table.
Denote by $\nu_\Delta$ a certain involutive graph automorphism of $\Delta$; we describe the non-identity ones by swapping of vertices as follows.
\[
\begin{array}{c||c|c|c|c|c|c|c|c}
\Delta     & \DynA_n & \mathbb{B}_n,\mathbb{C}_n & \DynD_n & \DynE_6 & \DynE_7 & \DynE_8 & \mathbb{F}_4 & \mathbb{G}_2 \\ \hline
h_\Delta   & n+1     & 2n                        & 2n-2    & 12      & 18      & 30      & 12 & 6 \\
\nu_\Delta & i \leftrightarrow n+1-i & \text{identity}&  n-1 \leftrightarrow n \text{, if $n$ odd} & 1\leftrightarrow 5 & \text{identity} & \text{identity}& \text{identity}& \text{identity}\\
 &  & &  \text{identity, if $n$ even} & 2\leftrightarrow 4 & & & & 
\end{array}
\]

Since $A$ is representation-finite, its associated valued quiver $Q$ is of Dynkin type $\Delta$.
For a vertex $i\in \Delta_0$, we fix $e_i$ to be the primitive idempotent of the path algebra of type $\Delta$.
Denote by $P_i:=e_iA$ (resp. $I_i:=D(Ae_i)$) the corresponding indecomposable projective (resp. injective) module.
We define $\ell_i$ the cardinality of the $\tau_A^-$-orbit of $P_i$.
Equivalently, we have
\[
\ell_i = 1+\max\{j\geq 1\mid s_{P_i}(j)=0\} = \min\{j\geq 1\mid s_{P_i}(j)=1\}.
\]
Note that $P_i^{\succ \ell_i-1}=\tau_A^{-(\ell_i-1)}(P_i)\cong I_{\nu_\Delta(i)}$.

The following fact is well-known; see, for example, \cite{MiyYek}.
\begin{lemma}\label{lem-h-Delta}
There is a natural isomorphism of auto-equivalences $\nu_{\D}^{-h_\Delta}\cong [2]$ on $\D$.
\end{lemma}
In particular, by combining with (\ref{eq-nuD}), we get that 
\begin{align}\label{eq-li-h}
\ell_{i}+\ell_{\nu_\Delta(i)}=h_\Delta.
\end{align}

\begin{lemma}\label{lem-RF0}
The following are equivalent.
\begin{enumerate}[(1)]
\item $\idim_{A^{(m)}}[P_i]_0=2m-s_{P_i}(m)$;
\item $\idim_{A^{(m)}}[P_i]_0=\domdim_{A^{(m)}}([P_i]_0)$;
\item $P_i^{\succ m}$ is injective;
\item $P_i^{\succ m} \cong I_i$ with $s_{P_i}(m)$ odd, or $P_i^{\succ m}\cong I_{\nu_\Delta(i)}$ with $s_{P_i}(m)$ even;
\item $m\in\{\ell_i-1, h_\Delta-1\}+h_\Delta\Z_{\geq 0}$
\end{enumerate}
Moreover, in such a case, we have
\begin{align}\label{eq-spi(m)-hered}
s_{P_i}(m) = \begin{cases}
2t, & \text{if $m=\ell_i-1+th_\Delta$ with $t\geq 0$;}\\
2t+1, & \text{if $m=h_\Delta-1+th_\Delta$ with $t\geq 0$,}
\end{cases}
\end{align}
\end{lemma}
\begin{proof}
\underline{(1)$\Leftrightarrow$(2)$\Leftrightarrow$(3)}: Follows from Proposition \ref{prop-dim}.

\underline{(3)$\Leftrightarrow$(4)}: Follows from the facts that $\nu_\Delta^2=\mathrm{id}$, $P_i^{\ell-1}\cong I_{\nu(i)}$, and $\nu_1^-(I_i)\cong P_i[1]$.

\underline{(4)$\Leftrightarrow$(5)}:
This follows from Lemma \ref{lem-h-Delta}; namely, we have the following equivalent conditions.
\begin{align}
m=\ell_i-1+th_\Delta\text{ for some }t\geq 0 & \Leftrightarrow \nu_1^{-m}(P_i)\cong \nu_1^{-(\ell_i-1)}(P_i)[2t] \cong I_{\nu_\Delta(i)}[2t] \in \add(DA[2t]), \notag \\
m=-1+th_\Delta\text{ for some }t\geq 1 & \Leftrightarrow \nu_1^{-m}(P_i)\cong \nu_1(P_i)[2t] \cong I_{i}[2t-1] \in \add(DA[tk-1]). \notag 
\end{align}

The last statement is clear from these equivalences.
\end{proof}

\begin{corollary}\label{cor-RF}
Suppose $A$ is a representation-finite hereditary algebra.
\begin{enumerate}[(1)]
\item If $\ell_i\neq \ell_j$ for some $i,j$, then $A^{(m)}$ being a higher Auslander algebra is equivalent to $m= th_\Delta-1$ for some $t>0$.
In such a case, we have 
\[
\domdim A^{(m)}=2t(h_\Delta-1)-1=\gldim A^{(m)}.
\]

\item If all $\ell_i$'s are equal, then $A^{(m)}$ being a higher Auslander algebra is equivalent to $m=th_\Delta/2-1$ for some $t>0$.
In such a case, we have
\[
\domdim A^{(m)}=t(h_\Delta-1)-1=\gldim A^{(m)}.
\]
\end{enumerate}
\end{corollary}
\begin{proof}
Note that $\domdim A^{(m)}=\idim A^{(m)}$ (the latter being equal to $\gldim A^{(m)}$) if, and only if, $\domdim_{A^{(m)}}[P_i]_0=\idim_{A^{(m)}}[P_i]_0$ for all $i\in Q_0$.
Therefore, in the case when equality holds, by (2)$\Rightarrow$(5) of Lemma \ref{lem-RF0}, $m$ has to be $th_\Delta-1$, or $\ell+th_\Delta-1$ when $\ell_i=\ell$ for all $i$.
Note that $\ell=h_\Delta/2$ in the latter case.
This proves one direction for both (1) and (2).

For the converse, if $m$ is as stated, then by (5)$\Rightarrow$(2) of Lemma \ref{lem-RF0}, $\idim_{A^{(m)}} [P_i]_0=\domdim_{A^{(m)}} [P_i]_0$ for all $i$, and this is equal to $2m-s_{P_i}(m)$.
By the formula (\ref{eq-spi(m)-hered}) in Lemma \ref{lem-RF0}, this value does not depend on $i$; hence, it is precisely the dominant dimension as well as the injective dimension of $A^{(m)}$.
\end{proof}
\begin{remark}
Note that the condition $\ell_i=\ell_j(=h_\Delta/2)$ for all $i,j$ is satisfied if, and only if, the orientation of $Q$ is stable under the involution $\nu$ (c.f. the description of the AR-quiver of $KQ$ from \cite{Gab1}).
A path algebra $KQ$ which satisfies this condition is called \emph{$(h_\Delta/2)$-homogeneous} in \cite{HerIya}.
\end{remark}

Combining Corollary \ref{cor-RI} and Corollary \ref{cor-RF}, we obtain a new characterisation of representation-(in)finiteness of a hereditary algebra.
\begin{corollary}\label{cor-summary-hered}
The following are equivalent for a hereditary algebra $A$.
\begin{itemize}
\item $A$ is representation-finite (resp. representation-infinite);
\item there exist infinitely many (resp. there does not exists any) $m\in \Z_+$ such that $A^{(m)}$ is a higher Auslander algebra.
\end{itemize}
Moreover, in the representation-finite case, the global and dominant dimension of the higher Auslander replicated algebra is equal to $t(h_\Delta-1)-1$ for some $t>0$.
\end{corollary}

\subsection{Further Properties}\label{subsec-hered-further}
The fact that the base algebra $A$ is hereditary allows us to obtain a few more properties about $A^{(m)}$ compared to the more general case of base algebras which will be studied in the next section.
The key property of $A$ being hereditary allows us list the indecomposable $A^{(m)}$-modules explicitly - they are all in the forms listed in Observation \ref{eg-modules}.

\begin{proposition}\label{prop-reptype}
Every indecomposable $A^{(m)}$-module is in one of the following forms.
\begin{itemize}
\item $[M]_i$ with $M$ an indecomposable $A$-module and $0\leq i\leq m$.
\item $[\nu_A(P),P]_i$ with $P$ an indecomposable projective $A$-module and $0\leq i<m$.
\item $\Omega_{A^{(m)}}^{-}([M]_i)$ with $M$ an indecomposable non-injective $A$-module and $0\leq i<m$.
\end{itemize}
In particular, $A^{(m)}$ preserves the representation-type of $A$ for all $m\geq 1$.
Moreover, if $A$ is representation-finite with $r$ indecomposable modules, then there are $(2m+1)r$ indecomposable $A^{(m)}$-modules.
\end{proposition}
\begin{proof}
According to \cite{ABST}, this is somewhat well known - as it is just a consequence of $A^{(m)}$ being a ``nice" idempotent truncation of the repetitive algebra $\hat{A}$.
The stated list can then be obtained using the theory of repetitive algebras.
\end{proof}

This result hinted that there should be a close connection between indecomposable $A^{(m)}$-modules and dimension vectors.
We closely follow the article \cite{PS} in the following.
Consider a basic algebra $A=KQ/I$ over an algebraically closed field $K$.
It is called \emph{triangular} if the quiver $Q$ contains no oriented cycles.
Given a triangular algebra $A$, the \emph{Tits form} $q_A(x)$ of $A$ is defined as 
\begin{align}
q_A(x) := \sum\limits_{i \in Q_0}^{}{x_i^2}- \sum\limits_{\alpha \in Q_1}{x_{s(\alpha)} x_{t(\alpha)}}+ \sum\limits_{i,j \in Q_0}{\dim_K \Ext_A^2(S_i,S_j)  x_i x_j},
\end{align}
where $x=(x_i)_{i\in Q_0}$, and $s(-), t(-)$ denote the start and target of an arrow respectively. 
The \emph{roots} of $q_A$ are the vectors $v \in \mathbb{Z}^{|Q_0|}$ such that $q_A(v)=1$.
A root $v$ of $q_A$ is \emph{positive} if all of its entries are non-negative.
Gabriel's famous theorem \cite{Gab2} says that for a representation-finite hereditary algebra, there is a one-to-one correspondence between the dimension vectors of the indecomposable modules and the positive roots of the Tits form.

It is not difficult to see that the same holds for $A^{(m)}$ when the base algebra $A$ is representation-finite hereditary, and we suspect this is known to people who have worked with replicated algebras.
Since we cannot find this precise statement in the literature, a brief explanation is provided below.
\begin{proposition}\label{prop-tits}
If $A$ is representation-finite hereditary algebra over an algebraically closed field $K$, then
$X\mapsto \mathbf{dim}(X)$ induces a bijection from the set of isoclasses of indecomposable $A^{(m)}$-modules to the set of of positive roots of the Tits form of $A^{(m)}$
\end{proposition}
\begin{proof}
It follows from the definition of $A^{(m)}$ and heredity of $A$ that $A^{(m)}$ is triangular; so it makes sense to talk about Tit forms of $A^{(m)}$.
Moreover, $A^{(m)}$ is a so-called convex subalgebra of the repetitive algebra $\hat{A}$, which is simply-connected.
This implies that $A^{(m)}$ is also simply-connected.

Recall from Proposition \ref{prop-reptype} that $A^{(m)}$ is representation-finite.
Hence, $A^{(m)}$ satisfies all the assumption required to apply a theorem of Bongartz (see \cite[Theorem 7.2]{PS}), which gives the claim.
\end{proof}

Recall the following generalisations of quasi-Frobenius (i.e. self-injective) algebras from \cite{Thr}.
\begin{definition}\label{def-QF123}
Let $A$ be a finite dimensional algebra.
\begin{itemize}
\item $A$ is \emph{QF-1} if every faithful left $A$-module $M$ has double centraliser property, i.e. there is a canonical ring epimorphism $A\to \End_{\Gamma}(M_{\Gamma})$, where $\Gamma:=\End_A(M)^{\mathrm{op}}$. 
\item $A$ is \emph{QF-2} if every indecomposable projective $A$-module has a simple socle.
\item $A$ is \emph{QF-3} if $\domdim (A)\geq 1$.
\end{itemize}
\end{definition}
Note that any self-injective algebra satisfies all three conditions.
An algebra is \emph{QF-13} if it is QF-1 and QF-3.

\begin{proposition}\label{prop-yamagata-QF13}{\rm \cite[Corollary 3.5.1]{Yam}}
Suppose $A$ is a QF-3 algebra with $\domdim A\geq 2$.
Then $A$ is also QF-1 if, and only if, every indecomposable module $M$ has $\codomdim(M)\geq 1$ or $\domdim (M)\geq 1$
\end{proposition}

\begin{proposition}\label{prop-QF}
Suppose $A$ is a hereditary algebra and $m$ is a positive integer.
For any indecomposable $A^{(m)}$-module $X$ which is not projective-injective, we have 
\begin{itemize}
\item $\domdim_{A^{(m)}}(X)\geq 1$ if, and only if, $X=[M]_i$ with $i<m$ for some indecomposable $A$-module, or $X=\Omega_{A^{(m)}}^-([M]_i)$ with $i<m$ for some indecomposable non-injective $A$-module $M$.
\item $\codomdim_{A^{(m)}}(X)\geq 1$ if, and only if, $X=[M]_i$ with $i>1$ for some indecomposable $A$-module, or $X=\Omega_{A^{(m)}}^-([M]_i)$ with $i\leq m$ for some indecomposable non-injective $A$-module $M$.
\end{itemize}
In particular, $A^{(m)}$ is a QF-13 algebra for all $m\geq 2$, or for all $m\geq 1$ when $A$ is not uniserial.
\end{proposition}
\begin{proof}
It is clear from Lemma \ref{lem-injcopres} that $\Omega_{A^{(m)}}^-([M]_i)$ has projective cover being injective, and injective envelope being projective, as long as $i<m$.

For $X=[M]_i$, Lemma \ref{lem-injcopres} shows that the injective envelop of $[M]_i$ is the projective-injective module $[I_M,\nu_A^-(I_M)]_{i}$ for all $i<m$.
Hence, it has dominant dimension at least 1.
On the other hand, $[M]_{m}$ has injective envelope $[I_M]_{m}$, which is not projective.

Likewise, the projective cover of $[M]_i$ with $i>0$ is the projective-injective module $[\nu_A(P_M),P_M]_{i-1}$, where $P_M$ is the projective cover of $M$ in $\mod A$; whereas $[M]_0$ has a non-injective projective cover $[P_M]_0$.

It follows from the discussion after Theorem \ref{thm-dims} that $A^{(m)}$ is QF-3 under the stated situations; hence, the final claim follows from Proposition \ref{prop-yamagata-QF13}.
\end{proof}
\begin{remark}
In general, $A^{(m)}$ is not a QF-2 algebra.
For example, if $A$ is the Kronecker algebra, then one can see from Example \ref{eg-typeA-Kronecker} (2) that $e_{2m+1}A^{(1)}$ has non-simple socle, namely $S_{2m+2}\oplus S_{2m+2}$.
\end{remark}

\section{Serre-formal algebras and their replicated algebras}\label{sec-Nu-formal}
The aim of this section is to introduce a new class of algebras which we call \emph{Serre-formal}, and give some basic properties.
We can naturally extend the methods and results from the previous section on hereditary algebras to this new class of algebras.

\subsection{Definition and basic properties}\label{subsec-NF-Bkgd}
We assume throughout this section that $A$ is a ring-indecomposable finite dimensional algebra.

Suppose $A$ is Iwanaga-Gorenstein.
Consider the bounded derived category $\Db(\mod A)$ of $A$-modules.
Denote by $\D$ the subcategory $\mathrm{perf}(A)$ given by the perfect complexes.
In this setting, we have endofunctors on $\D$ given by 
\[
\nu:= DA\Lotimes_A- = D\circ \RHom_A(-,A) \quad\text{ and }\quad \nu^{-1}:= \RHom_A(DA,-)
\]
which are quasi-inverses of each other; see \cite{Hap1}.
Note that $\nu$ is the Serre functor of $\D$.

\begin{definition}
An algebra $A$ is \emph{Serre-formal}, or \emph{$\nu$-formal}, if it is Iwanaga-Gorenstein and $\nu^{k}(A)$ is isomorphic to a direct sum of stalk complexes of $A$-module, i.e. $\nu^k(A)\cong \bigoplus_{n\in \Z} H^n(\nu^{k}(A))[-n]$, for all integers $k$, in $\Db(\mod A)$.
\end{definition}

The typical examples of $\nu$-formal algebras are self-injective algebras and hereditary algebras.
The Auslander-Reiten theoretic generalisation of hereditary algebras - $d$-hereditary algebras - are also examples of $\nu$-formal algebra; this can be seen immediately from the following definition.

\begin{definition}\label{def-dhered}
Let $d$ be a non-negative integer.
An algebra $A$ is \emph{$d$-hereditary} \cite{HIO} if $\gldim A=d$ and, for all $i\in \Z$, $\nu^i(A)$ belongs to $\add\bigcup_{k\in \Z}(\mod A)[dk]$.
\end{definition}

\begin{example}\label{eg-alg}
We give two easy examples, as well as a non-example.
\begin{enumerate}[(1)]
\item If $A$ is $\nu$-formal, then so is $A^{\mathrm{op}}$.  This follows from the following commutative diagram \[
\xymatrix@C=60pt{\mathrm{perf}(A)  \ar[r]^{\nu_{\mathrm{perf}(A)}} \ar@{<->}[d]_{D}^{\wr}& \mathrm{perf}(A)\ar@{<->}[d]_{D}^{\wr}\\
\mathrm{perf}(A^{\mathrm{op}}) \ar[r]^{\nu^{-1}_{\mathrm{perf}(A^{\mathrm{op}})}} & \mathrm{perf}(A^{\mathrm{op}}).
}
\]

\item The family of 1-Iwanaga-Gorenstein algebras arising from symmetrizable Cartan matrix introduced in \cite{GLS} are $\nu$-formal.
This follows from combining \cite[Prop 11.6]{GLS} and \cite[Theorem 1.1]{GLS}.
Namely, for such an algebra $A$, any preprojective and preinjective $A$-module (i.e. direct summands of $\tau^{k}(A)$ for some $k\in\mathbb{Z}$) has injective and projective dimension at most 1, which implies that $A$ is $\nu$-formal.
Note that this family contains all the path algebras associated to finite acyclic quivers as well as many other algebras of infinite global dimension.

\item (Non-example) A generalisation, called ``pro-species algebras", of the class of algebras in (2) was introduced  \cite{Kue}; the following example shows that such an algebra is not necessarily $\nu$-formal.

Let $A=KQ/I$ where
\[
Q:= \xymatrix@R=20pt{ 1 \ar[r]^{\alpha} & 2\ar@/^1pc/[r]^{\beta} & 3\ar@/^1pc/[l]^{\gamma} }
\]
and $I$ is generated by $\beta\gamma, \gamma\beta$.
One can check that this is a 1-Iwanaga-Gorenstein algebra; in fact, this is a ``locally self-injective" pro-species algebra (c.f. \cite{Kue}).
The minimal injective coresolution of $e_2A$ is $0\to D(Ae_3) \to D(Ae_1)\to 0$.
So we get that $\nu^{-1}(e_2A) = (e_3A\to e_1A)$, which has non-zero cohomology at both degree 0 and 1; hence, $A$ is not $\nu$-formal.
\end{enumerate}
\end{example}

Note that the class of $\nu$-formal algebras with finite global dimension is still strictly larger than that of $d$-hereditary algebras.
For example, the tensor product algebra of two hereditary algebras is not necessarily $d$-hereditary, but it is $\nu$-formal.
More generally, we have the following.
\begin{proposition}\label{prop-tensor}
If $A,A'$ are $\nu$-formal algebras, then $A\otimes_K A'$ is also $\nu$-formal.
\end{proposition}
\begin{proof}
Since $A,A'$ are both Iwanaga-Gorenstein, $A\otimes_K A'$ is also.
We will omit the subscript $K$ on the tensor product for clarity in the following.
For $X\in\mod A$ and $Y\in \mod A'$, we have isomorphisms
\begin{align}
\nu(X\otimes Y) & \cong D(A\otimes A')\Lotimes_{A\otimes A'} (X\otimes Y)  \cong (DA\Lotimes_A X)\otimes DA'\Lotimes_{A'}Y \cong \nu(X) \otimes \nu(Y),\notag 
\end{align}
where the last two $\nu$ are the auto-equivalences on $\mathrm{perf}(A)$ and $\mathrm{perf}(A')$ respectively.
This means that, for all $i\geq 0$, if $\nu^{i}(X)$ and $\nu^{i}(Y)$ are both direct sum of stalk complexes, then so is $\nu^{i}(X\otimes Y)$.
Hence, the required condition on $\nu^{i}(A\otimes A')$ is satisfied by induction on $i\geq 0$.

Similarly, the isomorphisms
\begin{align}
\nu^{-1}(X\otimes Y) & \cong \RHom_{A\otimes A'}(D(A\otimes A'), X\otimes Y)  \cong \RHom_{A\otimes A'}(DA\otimes DA', X\otimes Y) \notag \\
& \cong \RHom_A(DA,X)\otimes \RHom_{A'}(DA',Y) \cong \nu^{-1}(X) \otimes \nu^{-1}(Y)\notag 
\end{align}
yields $\nu^{-i}(A\otimes A')\cong \nu^{-i}(A)\otimes \nu^{-i}(A')$ for all $i\geq 0$, which shows that the required condition on $\nu^{-i}(A\otimes A')$ is satisfied.
\end{proof}

We have particular interest in the modules $\bigoplus_{k\in\Z}\nu^{k}(A)$ for a $\nu$-formal algebra $A$.
If one bears in mind that the toy model for $\nu$-formal algebras are the hereditary algebras, then one can think of them as analogue of preprojective and preinjective modules.

For ease of further exposition, let us fix some notations.
We denote by $\Simp$ the indexing set of (isomorphism classes of) simple $A$-modules.
The indecomposable projective (resp. injective) $A$-module corresponding to $x\in \Simp$ is denoted by $P_x$ (resp. $I_x=\nu_A(P_x)$).

For $x\in \Simp$ and $k\in\Z_{\geq 0}$, we have 
\[
\nu^{-k}(P_x)\in (\mod A)[s_x^-(k)]\quad \text{ and }\quad \nu^{k}(I_x)\in (\mod A)[s_x^+(k)]
\]
for some integers $s_x^{\pm}(k)$.
Note that $s_x^-$ and $s_x^+$ are decreasing and increasing functions on $\Z_{\geq 0}$ respectively with $s_x^-(0)=s_x^+(0)=0$.
We define $s_x^-$ and $s_x^+$ separately as it is more convenient to state our forthcoming results in this way.

We will denote the indecomposable $A$-modules $\nu^{-k}(P_x)[-s_x^-(k)]$ and $\nu^{k}(I_x)[-s_x^+(k)]$ by $P_x^{\succ k}$ and $I_x^{\prec k}$ respectively.
Note that as $\nu$ is an auto-equivalence on $\D=\mathrm{perf}(A)$, these modules have finite projective dimensions. 
By definitions, we have
\begin{align}
& P_x^{\succ k+1} = \nu^{-(k+1)}P_x[-s_x^-(k+1)] \cong \nu^{-1}P_x^{\succ k}[s_x^-(k)-s_x^-(k+1)],\notag \\
\text{and }\quad & I_x^{\prec k+1} = \nu^{k+1}I_i[-s_x^+(k+1)] \cong \nu I_x^{\prec k}[s_x^+(k)-s_x^+(k+1)].\notag 
\end{align}
Our notations are chosen so that for a hereditary algebra $A$, $P_x^{\succ k+1}$ lies on the right hand side of $P_x^{\succ k}$ on the Auslander-Reiten quiver of $\D$.

\begin{lemma}\label{lem-vanish}
Let $A$ be a $\nu$-formal algebra and fix $x\in \mathcal{S}$.
Then the following holds for all $k\in \Z_{\geq 0}$.
\begin{itemize}
\item[(1)] $\Ext_A^p(DA,P_x^{\succ k})\neq 0$ is equivalent to $p=s_x^-(k)-s_x^-(k+1)$;

\item[(1')] $\Ext_A^p(I_x^{\prec k},A)\neq 0$ is equivalent to $\mathrm{Tor}_p^A(DA,I_x^{\prec k})\neq 0$, and also to $p=s_x^+(k+1)-s_x^+(k)$;

\item[(2)] $\idim_{A} P_x^{\succ k} = s_x^-(k)-s_x^-(k+1) = \pdim_{A} P_x^{\succ k+1}$;
\item[(2')] $\pdim_{A} I_x^{\prec k} = s_x^+(k+1)-s_x^+(k) = \idim_{A} I_x^{\prec k+1}$;
\end{itemize}
In particular, $P_x^{\succ k}$ is injective (resp. projective) if, and only if, $s_x^-(k)=s_x^-(k+1)$ (resp. $s_x^-(k)=s_x^-(k-1)$).
Dually, $I_x^{\prec k}$ is projective (resp. injective) if, and only if, $s_x^+(k)=s_x^+(k+1)$ (resp. $s_x^+(k)=s_x^+(k-1)$).
\end{lemma}
\begin{proof}
Since $\Ext_A^p(DA,P_x^{\succ k})=H^p(\RHom_A(DA,P_x^{\succ k}))=H^p(\nu^{-1}(P_x^{\succ k}))$ for all $p\in\Z$, $\nu$-formality of $A$ implies that $\Ext_A^p(DA,P_x^{\succ k})$ vanishes for all indices but $p=s_x^-(k)-s_x^-(k+1)$.
This gives us (1).
(1') can be proved similarly by considering the cohomology of $DA\Lotimes_A-$, and the functorial isomorphism $D\Ext_A^p(-,A)\cong \mathrm{Tor}_p^A(DA,-)$.

We prove (2) by induction on $k$.
The case of $k=0$ is trivial.

Suppose the claim holds for $P_x^{\succ k}$ for some $k\geq 0$.
If $P_x^{\succ k}$ is injective, then $\nu^{-1}(P_x^{\succ k})=\nu_A^-(P_x^{\succ k})\in \proj A$ and we are done.

Thus, we can assume $P_x^{\succ k}$ is non-injective.
Since $A$ is Iwanaga-Gorenstein, it follows from Lemma \ref{lem-IGfacts} that finiteness of $\pdim_AP_x^{\succ k}$ implies that of $\idim_A P_x^{\succ k}$.
It now suffices to prove that $\pdim_A P_x^{\succ k}<\infty$, as the vanishing of the Ext-spaces forces $\idim_A P_x^{\succ k}=s_x^-(k)-s_x^-(k+1)$.

By induction hypothesis $P_x^{\succ k}$ has a minimal injective coresolution 
\[
0\to P_x^{\succ k} \to I^0 \to I^1 \to \cdots  \to I^p\to 0,
\]
where $p=s_x^-(k)-s_x^-(k+1)$.
Applying $\nu_A^{-}=\Hom_A(DA,-)$ yields a sequence
\[
0 \to \nu_A^-(P_x^{\succ k}) \to \nu_A^-(I^0) \to \nu_A^-(I^1) \to \cdots \to \nu_A^-(I^p) \to M \to 0,
\]
where $M=\mathrm{Coker}(\nu_A^-(f))$ and $f$ is the differential $I^{p-1}\to I^p$.
By construction, we have $M=H^p(\nu^{-1}(P_x^{\succ k}))$, which means that $M\cong P_x^{\succ k+1}$ and $p=s_x^-(k)-s_x^-(k+1)$.
It now follows from (1) that the sequence above is exact.
In particular, it is the minimal projective resolution of $P_x^{\succ k+1}$.
Hence, we have $\pdim_A P_x^{\succ k+1} = s_x^-(k)-s_x^-(k+1)$ as required.

(2') can be proved analogously, which we leave the reader to check.
\end{proof}

\begin{example}
When $A$ is $d$-hereditary, an indecomposable $A$-module $M$ is called $d$-preprojective (resp. $d$-preinjective) if $M=\nu^{k}(P_x)[-dk]$ for some integer $k\leq 0$ (resp. $k\geq 1$).
Moreover, we have
\begin{align}
\idim_A P_x^{\succ k} &= s_x^-(k)-s_x^-(k+1)=\begin{cases} d, & \text{if $P_x^{\succ k}$ is non-injective;} \\ 0, & \text{otherwise.}\end{cases}\notag \\
\pdim_A I_x^{\prec k} &= s_x^+(k+1)-s_x^+(k)=\begin{cases} d, & \text{if $I_x^{\prec k}$ is non-injective;} \\ 0, & \text{otherwise.}\end{cases}\notag 
\end{align}
\end{example}

Having Theorem \ref{thm-dims} in mind, it is natural to ask what subclass of the $\nu$-formal algebras play the analogous role of representation-finite hereditary algebras.
\begin{definition}
Let $A$ be a $\nu$-formal algebra and $x\in \Simp$.
Define
\[
\ell_x := \inf \{ k\geq 0 \mid P_x^{\succ k}\in \add A \} = \inf\{ k\geq 1 \mid P_x^{\succ k-1}\in \add DA\}.
\]
We say that $A$ is \emph{periodically $\nu$-formal} if $\ell_x<\infty$ for all $x\in \Simp$; otherwise,  \emph{aperiodically $\nu$-formal}.
In the former case, we define $\sigma$ to be the permutation on $\Simp$ so that $P_x^{\succ \ell_x}\cong P_{\sigma(x)}$, and $k_A$ to be its order.
\end{definition}

By Lemma \ref{lem-vanish}, $A$ being periodically $\nu$-formal is the same as the existence of $k_x\in \Z_{\geq 0}$ so that $s_x^-(k_x)=s_x^-(k_x+1)$ (or $s_x^+(k_x)=s_x^+(k_x+1)$) for all $x\in \Simp$.
Note also that $\ell_x$ can also be defined in terms of $I_x^{\prec k}$ instead, namely
\[
\ell_x = \inf \{ k\geq 0\mid I_{\sigma(x)}^{\prec k}\in \add DA \} = \inf\{ k\geq 1\mid I_{\sigma(x)}^{\prec k-1}\in \add A\}.
\]

Recall that if $\phi$ is an automorphism on $A$, then there is a natural auto-equivalence $\phi^*$ on $\mod A$ (hence, on $\mathrm{perf}(A)$) given by twisting the $A$-action on the modules by $\phi$.

\begin{definition}
Following \cite{HerIya}, an algebra $A$ is said to be \emph{twisted fractionally Calabi-Yau} if there exist integers $n$ and $m\neq 0$, as well as an automorphism $\phi$ of $A$, so that there is a natural isomorphism of auto-equivalences $\nu^m\simeq [n]\circ \phi^*$ on $\mathrm{perf}(A)$.
It follows from \cite[Prop 4.3]{HerIya} that this is equivalent to $\nu^m(A)\cong A[n]$ in $\Db(\mod A)$.
We say that the \emph{twisted Calabi-Yau dimension} of $A$ is $(n,m)$ if
$m$ is the minimal positive integer so that $\nu^m(A)\cong A[n]$ in $\mathrm{perf}(A)$.
\end{definition}

The following result generalises \cite[Theorem 1.1]{HerIya}.

\begin{proposition}\label{prop-twistedCY}
Suppose $A$ is a $\nu$-formal algebra.
Then $A$ is periodically $\nu$-formal if, and only if, $A$ is twisted fractionally Calabi-Yau.
\end{proposition}
\begin{proof}
The if-part follows from the definition of twisted fractionally Calabi-Yau.
We concentrate on showing the converse direction in the rest.

Let us first define $h_x:= \ell_x + \ell_{\sigma(x)} + \cdots +\ell_{\sigma^{k_A-1}(x)}$ and $p_x:=s_x^-(h_x)$.
We claim that the rational number $p_x/h_x$ is independent of $x\in \Simp$.
Indeed, this follows from a simple modification of the argument in \cite[subsection 4.1]{HerIya}, which we explain below for completeness.

Suppose $x,y\in \Simp$ are chosen so that $\Hom_A(P_x,P_y)\neq 0$.
Since 
\[
\nu^{-h_x}(P_x) \cong P_{\sigma^{k_A}(x)}[s_x^-(h_x)] = P_x[p_x]
\]
by definition, we get that
\[
0 \neq \Hom_{\D}\big(\nu^{-h_yh_x}(P_x),\nu^{-h_xh_y}(P_y)\big) \cong \Hom_{\D}(P_x[h_yp_x],P_y[h_xp_y]).
\]
This implies that $h_yp_x = h_xp_y$.
It follows from the assumption of $A$ being ring-indecomposable that $p_x/h_x=p_y/h_y$ for all $x,y\in\Simp$.

Take $m$ to be the least common multiple of $h_x$'s over all $x\in\Simp$, then it suffices to show that $s_x^-(m)$ is independent of $x\in\Simp$ to finish the proof.
Indeed, since the isomorphism $\nu^{-m}(P_x)=\nu^{-(m/h_x)h_x}(P_x)\cong P_x[(m/h_x)p_x]$ holds for all $x\in \Simp$, we have $s_x^-(m)=mp_x/h_x$ for all $i\in\Simp$.
It now follows from the above claim on $p_x/h_x$ that $s:=s_x^-(m)$ is independent of $x\in\Simp$.
Hence, we have $\nu^{-m}(A)\cong A[s]$.
\end{proof}

Suppose $A$ is periodically $\nu$-formal.
We denote by $(h_A,-s_A)$ the twisted Calabi-Yau dimension of $A$.
Note that since $\nu$ is an auto-equivalence, $\nu^{-k}(A)\cong A[s]$ is equivalent to $\nu^k(A)\cong A[-s]$, so $h_A$ is equal to the minimal positive integer $k$ so that $\nu^{-k}(A)\cong A[s]$, and $s_A=s_x^-(h_A)=-s_x^+(h_A)$ for any $x\in\Simp$.

\begin{example}
If $A=kQ$ is representation-finite hereditary of Dynkin type $\Delta$, then $\sigma$ is the permutation on $Q_0$ induced by the graph automorphism $\nu_\Delta$ in the notation of the previous section.
This implies that $k_A=1$ or $k_A=2$, and the integer $m$ in the proof of Proposition \ref{prop-twistedCY} is $k_Ah_\Delta/2$.
If the orientation of $Q$ is stable under $\nu_\Delta$, then the twisted Calabi-Yau dimension is $(h_A,-s_A)=(h_\Delta/2,h_\Delta/2-1)$; otherwise, $(h_A,-s_A)=(h_\Delta,h_\Delta-2)$.
In particular, both $m$ and $h_A$ are not necessarily the Coxeter number associated to $A$.
\end{example}

We also remark that for periodically $\nu$-formal algebras $A$, there is no formula to express $k_A$, or even for the upper bound of $k_A$, in terms of $\idim A$ when $A$ is not hereditary.
For example, consider the algebra $A:= K\vec\DynA_n/I$, where $I$ is generated by the unique path of length $n-1$.
This is a 2-hereditary (and 2-representation-finite, see Subsection \ref{subsec-NF-naka}) algebra where $\sigma$ has order $k_A=n$ if $n$ odd; otherwise, $k_A=n/2$.

\subsection{Homological dimensions}\label{subsec-NF-dims}
We are now going to state the $\nu$-formal analogue of the results and their ``dual" of the previous section.
We will not give detailed proofs because they are exactly the same as in the hereditary case, but we will explain what extra ingredients and tweaks are needed in order to copy the proof directly.
The proof for the dual results (projective resolutions and projective dimensions of modules) will be completely omitted.
Moreover, our result in the previous section says that appearance of higher Auslander algebras as replicated algebras is equivalent to the base algebra being of Dynkin type.
The second aim of this section is to explain why the analogous result also holds for $\nu$-formal base algebras, i.e. the appearance of minimal Auslander-Gorenstein algebras as replicated algebras is equivalent to the base algebra being periodically $\nu$-formal.

Throughout, we assume $A$ is an indecomposable $\nu$-formal algebra with $\Simp$ indexing the simple $A$-modules.
Let us begin with the following analogue of Lemma \ref{lem-injcopres}.

\begin{lemma}\label{lem-injcopres2}
Let $m$ be a positive integer, $x\in \Simp$, and $k\geq 0$.
\begin{enumerate}
\item For any $0\leq i< m$, we have an exact sequence of $A^{(m)}$-modules:
\begin{align}
\xi(x,k,i):\qquad  0 \to [P_x^{\succ k}]_i \to I^0 \to I^1 \to \cdots \to I^p \to [P_x^{\succ k+1}]_{i+1} \to 0,\notag
\end{align}
where $p=s_x^-(k)-s_x^-(k+1)$ and the middle $p+1$ terms are projective-injective.
This sequence is a truncation of the minimal injective coresolution of $[P_x^{\succ k}]_i$ in $\mod A^{(m)}$.

\item For any $0< i\leq  m$, we have an exact sequence of $A^{(m)}$-modules:
\begin{align}
\eta(x,k,i):\qquad  0 \to [I_x^{\prec k+1}]_{i-1} \to P_r \to \cdots \to P_1 \to P_0 \to [I_x^{\prec k}]_{i} \to 0,\notag
\end{align}
where $r=s_x^+(k+1)-s_x^+(k)$ and the middle $r+1$ terms are projective-injective.
This sequence is a truncation of the minimal projective coresolution of $[I_x^{\prec k}]_i$ in $\mod A^{(m)}$.
\end{enumerate}
\end{lemma}
\begin{proof}
As in the proof of Lemma \ref{lem-injcopres}, we take the minimal injective coresolution of $P_x^{\succ k}$ in $\mod A$, say, $0\to P_x^{\succ k}\to \overline{I}^0 \to \cdots \to \overline{I}^p\to 0$ where $p=s_x^-(k)-s_x^-(k+1)$, and take $I^j := [\overline{I}^j,\nu_A^-(\overline{I}^j)]_i$ for all $0\leq j\leq p$.
The exactness at $[P_x^{\succ k}]_i$ and at $I^j$ for $0<j<p$ follows from Lemma \ref{lem-vanish} (1).
The exactness at the last two non-zero terms follows from the fact that $P_x^{\succ k+1}\cong \nu^{-1}(P_x^{\succ k})[s_x^-(k)-s_x^-(k+1)]$.
\end{proof}

Similarly, we can define the analogues $\xi(x,k,m)$ of (\ref{eq-seq3}), which will be an exact sequence containing $s_x^-(k)-s_x^-(k+1)$ terms of injective non-projective $A^{(m)}$-modules.
The minimal injective coresolution of the $A^{(m)}$-module $[P_x^{\succ k}]_i$ can now be obtained by iteratively taking the Yoneda products of 
\[
\xi(x,k,i), \xi(x,k+1,i+1), \cdots, \xi(x,k+m-1-i,m-1), \xi(x,k+m-i,m).
\]

Dually, we can define an exact sequence $\eta(x,k,0)$ which has $s_x^+(k+1)-s_x^+(k)+1$ terms of projective non-injective $A^{(m)}$-modules, and the minimal projective resolution of the $A^{(m)}$-module $[I_x^{\prec k}]_i$ can be obtained by iteratively taking the Yoneda products of
\[
\eta(x,k+i,0), \eta(x,k+i-1,1), \cdots, \eta(x,k+1,i-1), \eta(x,k,i).
\]

This gives the following analogue of Proposition \ref{prop-dim}; note that $s_x^-(k)$ here is equal to $s_{P_x}(k)-k$ in the notation of the previous section.

\begin{proposition}\label{prop-dim2}
Let $m$ be a positive integer, $x\in \Simp$, and $k\geq 0$.
Then we have the following formulae for all $0\leq i\leq m$.
\begin{align}
\domdim_{A^{(m)}}[P_x^{\succ k}]_i &= m-i+s_x^-(k)-s_x^-(k+m-i), \notag\\
\idim_{A^{(m)}}[P_x^{\succ k}]_i &= m-i+s_x^-(k)-s_x^-(k+m-i+1), \notag \\
\codomdim_{A^{(m)}}[I_x^{\prec k}]_i &= i+s_x^+(k+i)-s_x^+(k), \notag\\
\pdim_{A^{(m)}}[I_x^{\prec k}]_i &= i+s_x^+(k+i+1)-s_x^+(k). \notag 
\end{align}
\end{proposition}
\begin{proof}
We prove the first two formulae; the later two are proved dually.

For each $0\leq j\leq  m-i$, the sequence $\xi(x,k+j,i+j)$ contributes precisely $s_x^-(k+j)-s_x^-(k+j+1)+1$ injective terms in the minimal injective coresolution of $[P_x^{\succ k}]_i$.
Therefore, the injective dimension is given by 
\[
-1+\sum_{j=0}^{m-i}\left( s_x^-(k+j)-s_x^-(k+j+1)+1\right) = m-i + s_x^-(k) - s_x^-(k+m-i+1).
\]
For each $j<m-i$, all of the middle terms in $\xi(P,k+j,i+j)$ are projective(-injective); whereas for the case when $j=m-i$, none of them are projective.
Therefore, the dominant dimension is just $\idim_{A^{(m)}} [P_x^{\succ k}]_i-(s_x^-(k+m-i)-s_x^-(k+m-i+1))$.
\end{proof}

The following is our main result concerning the replicated algebras of $\nu$-formal algebras.
In the special case of $d$-hereditary algebras, due to the regularity in the behaviour of the functions $s_i^\pm$, the corresponding results give nicer formulae.
For the convenience of reader, we will state them separately in subsection \ref{subsec-NF-d-hered}.

\begin{theorem}\label{thm-dims2}
Let $A$ be a $\nu$-formal algebra, then for all $m\geq 1$, we have
\begin{align*}
\codomdim DA^{(m)} = \domdim A^{(m)} & = m-\max\{s_x^-(m)\mid x\in \Simp\} = m+\min\{s_x^+(m) \mid x\in \Simp\}, \\
\text{ and } \quad \pdim DA^{(m)} =  \idim A^{(m)} &= m-\min\{s_x^-(m+1)\mid x\in\Simp\} = m+\max\{s_x^+(m+1)\mid x\in\Simp\}.
\end{align*}
In particular, $A^{(m)}$ is Iwanaga-Gorenstein.
\end{theorem}
\begin{proof}
By definition, we have
\begin{align}
\domdim A^{(m)} & = \min\{ \domdim [P_x]_0 \mid x\in\Simp \}, & \idim A^{(m)} &= \max\{ \idim [P_x]_0 \mid x\in\Simp \},\notag \\
\codomdim DA^{(m)} & = \min\{ \codomdim [I_x]_0 \mid x\in\Simp \},  & \pdim DA^{(m)} &= \max\{ \pdim [I_x]_0 \mid x\in\Simp \}.\notag 
\end{align}
Recall (see Section \ref{sec-Prelim}) that for any algebra $\Lambda$, we always have $\domdim \Lambda=\codomdim D\Lambda$, so the two values in the first column are equal; similarly, we always have $\idim \Lambda = \pdim D\Lambda$ whenever both of them are finite, which gives the equality of the two values in the second column.
Now the claim follows from applying Proposition \ref{prop-dim2} and the fact that $s_x^\pm(0)=0$ for all $x\in\Simp$.
\end{proof}

One easy consequence of this theorem is that
\[
\min\{-s_x^-(k)\mid x\in\Simp\} = \min\{s_x^+(k)|x\in\Simp\} \quad \text{and}\quad \max \{ -s_x^-(k) \mid x\in\Simp\} = \max\{ s_x^+(k)\mid x\in\Simp\}
\]
for all $k\geq 0$.

\begin{corollary}\label{cor-minAG}
Let $A$ be a $\nu$-formal algebra and fix some $m\geq 1$.
Then the following are equivalent
\begin{enumerate}[(1)]
\item the $m$-th replicated algebra $A^{(m)}$ is minimal Auslander-Gorenstein;
\item $s_x^-(m)=s_y^-(m+1)$ for all $x,y\in\Simp$;
\item $P_x^{\succ m}$ is injective for all $x\in\Simp$;
\item $s_x^+(m)=s_y^+(m+1)$ for all $x,y\in\Simp$;
\item $I_x^{\succ m}$ is projective for all $x\in\Simp$;
\item $\nu^{-(m+1)}(A)\cong A[n]$ for some $n\in\mathbb{Z}$.
\end{enumerate}
\end{corollary}
\begin{proof}
By Theorem \ref{thm-dims2}, we have equivalence between (1) and any of the following conditions.
\begin{enumerate}[(i)]
\item $\max\{s_x^-(m)\mid x\in \Simp\}=\min\{s_x^-(m+1)\mid x\in \Simp\}$.
\item $\min\{s_x^+(m)\mid x\in \Simp\}=\max\{s_x^+(m+1)\mid x\in \Simp\}$.
\end{enumerate}
Since $s_x^-$ is a decreasing function over $\Z_{\geq 0}$ for each $x\in \Simp$, we have
\[
\max\{s_x^-(m)\mid x\in\Simp\} \geq s_y^-(m) \geq s_y^-(m+1) \geq \min\{s_x^-(m+1)\mid x\in\Simp\},
\]
for all $y\in \Simp$.
Therefore, (i) is equivalent to the condition (2), which in turn is equivalent to (3).
Dually, (ii) is equivalent to (4), which in turn is also equivalent to (5).

Given (6), then (2) and (3) are satisfied with $s_x^-(m)=n=s_y^-(m+1)$ for all $x,y\in\Simp$.
Conversely, given (2) (and hence, (3)), then clearly we get (6) with $n=s_x^-(m)$.
\end{proof}

\begin{corollary}\label{cor-minAG2}
Suppose $A$ is a $\nu$-formal algebra.
Then the following are equivalent.
\begin{enumerate}[(1)]
\item $A$ is periodically $\nu$-formal.
\item There exists some $m\in \Z_+$ such that $A^{(m)}$ is minimal Auslander-Gorenstein.
\item There exists infinitely many $m\in \Z_+$ such that $A^{(m)}$ is minimal Auslander-Gorenstein.
\end{enumerate}
Moreover, when $A$ is periodically $\nu$-formal with twisted Calabi-Yau dimension $(h_A,-s_A)$, $A^{(m)}$ is minimal Auslander-Gorenstein if, and only if, $m=th_A-1$ for some $t>0$; in such a case, we have
\[
\domdim A^{(m)} = t(h_A-s_A)-1 = \idim A^{(m)}.
\]
\end{corollary}
\begin{proof}
The equivalence is just an immediate consequence of Proposition \ref{prop-twistedCY} and Corollary \ref{cor-minAG}.
By the minimality of $h_A$, $\nu^{-m}(A)\cong A[n]$ is equivalent to $m=th_A$ for some integer $t>0$; the final equivalence follows.
Moreover, since $\nu^{-th_A}(A)\cong A[ts_A]$, we have $s_x^-(th_A-1)=s_x^-(th_A)=ts_A$.
Applying this to Theorem \ref{thm-dims2} yields $\idim A^{(m)}=\domdim A^{(m)}=th_A-1-ts_A$ as stated.
\end{proof}

\subsection{The case of higher hereditary algebras}\label{subsec-NF-d-hered}
We specialise Theorem \ref{thm-dims2} and its corollaries for the case when $A$ is a ring-indecomposable non-simple $d$-hereditary algebra with $d\geq 1$, i.e. $\gldim A=d$ and $\nu^i(A) \in \add\bigcup_{k\in \Z}(\mod A)[dk]$ for all $i\in \Z$.
Recall the following subclasses of $d$-hereditary algebras.

\begin{definition}\label{def-d-hered}
Let $d$ be a positive integer.
\begin{itemize}
\item An algebra $A$ is \emph{$d$-representation-finite} if $\gldim A\leq d$ and for each indecomposable projective $A$-module $P$, there is some $t_P\geq 0$ so that $\nu^{-t_P}(P[dt_P])$ is an indecomposable injective $A$-module.

\item An algebra $A$ is \emph{$d$-representation-infinite} if $\gldim A\leq d$ and $\nu^{-i}(A[di]) \in \mod A$ for all $i\geq 0$.
\end{itemize}
\end{definition}
In order to help comparing these algebras with the more general $\nu$-formal algebras, the definitions given here are not presented in the same way as the usual one in the literature.
Note that the usual definition of $d$-representation-finite algebra is stated in terms of the existence of $d$-cluster-tilting module and the requirement on global dimension; the two definitions are equivalent by \cite[Thm 1.23]{Iya2}.

It is clear from the definition that $d$-representation-finite algebras are periodically $\nu$-formal.
More generally, as shown in \cite[Prop 2.9, 3.3]{HIO}, any $d$-hereditary algebra is also $\nu$-formal.
Moreover, the statement of \cite[Lemma 3.5]{HIO} says precisely that $d$-representation-infinite algebras are aperiodically $\nu$-formal.
The following dichotomy theorem says that there does not exists any aperiodically $\nu$-formal algebra that is not $d$-representation-infinite.

\begin{theorem}[Dichotomy Theorem]{\rm \cite[Theorem 3.4]{HIO}}
Any ring-indecomposable $d$-hereditary algebra is either $d$-representation-finite or $d$-representation-infinite.
\end{theorem}

From now on, we assume that $A$ is a ring-indecomposable non-simple $d$-hereditary algebra.
As before, we use $\Simp$ to denote the indexing set of simple $A$-modules.
For $x\in \Simp$, the functions $s_x^\pm : \Z_{\geq 0}\to \pm \Z_{\geq 0}$ can be described more concretely relative to general $\nu$-formal algebras as follows.
\[
s_x^\pm(k) = \pm \left(k-r_x^\pm(k)\right)d,
\]
where $r_x^-(k)$ (resp. $r_x^+(k)$) is the number of (indecomposable) injective (resp. projective) modules appearing in the sequence $(P_x^{\succ i})_{0\leq i<k}$ (resp. $(I_x^{\prec i})_{0\leq i<k}$).
If $A$ is, moreover, $d$-representation-infinite, then $r_x^\pm(k)=0$ for all $x\in\Simp$ and $k\geq 0$.
For simplicity, we only consider the results involved on the $s_x^-$ side, since the other side can be treated dually.

Now Theorem \ref{thm-dims2} gives us the following.
\begin{corollary}\label{cor-dim-dhered}
If $A$ is an indecomposable non-simple $d$-hereditary algebra, then we have 
\begin{align}
\idim A^{(m)} = \pdim DA^{(m)} &= (d+1)m+d-(\min\{r_x^-(m+1)\mid x\in\Simp\})d,\notag \\
\text{and}\quad \domdim A^{(m)} = \codomdim DA^{(m)} & = (d+1)m-(\max\{r_x^-(m)\mid x\in\Simp\})d. \notag
\end{align}
In particular, if $A$ is $d$-representation-infinite, then
\[
\gldim  A^{(m)} = \idim  A^{(m)} = (d+1)m+d \quad \text{ and }\quad \domdim  A^{(m)} = \codomdim DA^{(m)} = (d+1)m.
\]
\end{corollary}

By Proposition \ref{prop-twistedCY} (or \cite[Theorem 1.1]{HerIya}), $A$ being $d$-representation-finite means that it is twisted fractionally Calabi-Yau of dimension $(h_A,-s_A)$.
Note that $h_A$ is given by $\ell_x+\ell_{\sigma(x)}+\cdots +\ell_{\sigma^{r-1}(x)}$ for some $0<r\leq k_A$ independent of $x\in\Simp$, and $s_A=s_x^-(h_A)=(r_x^-(h_A)-h_A)d=(r-h_A)d$.
Therefore, we have the following $d$-hereditary specialisation of Corollary \ref{cor-minAG2}.

\begin{proposition}\label{prop-RF-d}
Suppose $A$ is a ring-indecomposable non-simple $d$-representation-finite algebra of twisted Calabi-Yau dimension $(h,(h-r)d)$.
$A^{(m)}$ is higher Auslander algebra if, and only if, $m=th-1$ for some $t>0$.
Moreover, in such a case, we have 
\[  \gldim A^{(m)} = ((d+1)h-dr)t-1 = \domdim A^{(m)} \]
\end{proposition}

\subsection{Further example I: Quotient of linear oriented path algebra of type A}\label{subsec-NF-naka}
Consider the quiver $\vec\DynA_n$ given by linearly oriented Dynkin quiver of type $\DynA_n$ (with $n>1$):
\[
\vec\DynA_n:= \xymatrix{ 1 \ar[r] & 2 \ar[r] & 3 \ar[r] & \cdots \ar[r]& {n-1} \ar[r] & {n}}.
\]
In this subsection, we will classify all the ring-indecomposable quotient algebras of $K\vec\DynA_n$ which are $\nu$-formal.
It turns out that $\nu$-formality is equivalent to higher representation-finiteness.
Let us recall some background information and fix some notations first.

Since $\vec\DynA_n$ is acyclic, the global dimension of any quotient $A$ of $K\vec\DynA_n$ is always finite.
In particular, it is Iwanaga-Gorenstein with $\gldim A = \idim A_A$.

Our calculation depends only on a combinatorial tool which uniquely determines the Morita equivalence classes of Nakayama algebras - the Kupisch series.
Suppose $A$ is a Nakyama algebra with (Ext-)quiver $Q$, then its \emph{Kupisch series} $[c_1,c_2,\ldots ,c_n]$ is given by $c_i$ being the length of the indecomposable projective $A$-module corresponding to the vertex $i\in Q_0$.
For example, the algebra $K\vec\DynA_n$ has Kupisch series $[n,n-1,n-2,\ldots ,2,1]$.
In our setting, since $A$ is a ring-indecomposable quotient of $\DynA_n$, the Kupisch series always satisfies the conditions $c_n=1$, $c_i \geq 2$ for $i < n$, and $c_i\leq c_{i+1}+1$.

Throughout, we will use $P_i, I_i, S_i$ to denote the indecomposable projective, injective, simple modules corresponding to $1\leq i\leq n$ respectively.
For $l\leq n$, we define $T_{n,l}:=K\vec\DynA_n/\rad^l(K\vec\DynA_n)$ to be the Nakayama algebra with $n$ simple modules and Kupisch series $[l,l,l,\ldots,l,l-1,l-2,\ldots,2,1]$.
For the background on the computation of minimal projective resolutions and minimal injective coresolutions in Nakayama algebras, see for example \cite{Mar}.

\begin{lemma}\label{lem-naka-3types}
Suppose $A$ is a ring-indecomposable quotient algebra of $K\vec\DynA_n$ with Kupisch series $[c_1,c_2,\ldots, c_n]$.  Then one of the following holds.
\begin{enumerate}[(1)]
\item There exists $1\leq i<n$ such that $c_i<c_{i+1}$.
\item There exists $1< i<n$ such that $c_{i-1}-1=c_i=c_{i+1}$.
\item There exists $l\leq n$ such that $A=T_{n,l}$.
\end{enumerate}
\end{lemma}
\begin{proof}
If (1) does not hold, then we have $c_1 \geq c_2 \geq \cdots \geq c_n=1$.
Since (3) is equivalent to having some $1\leq i<n$ with $c_1=c_2=\cdots =c_i>c_{i+1} >\cdots >c_n$, when both (1) and (3) fail to hold, then there exists some $1<i<n$ with $c_{i-1}>c_i=c_{i+1}$, which is exactly (2).
\end{proof}

\begin{lemma}\label{lem-non-nf-naka1}
Suppose $A$ is a ring-indecomposable quotient algebra of $K\vec\DynA_n$.
If there exists $1\leq i<n$ so that $c_i<c_{i+1}$, then $A$ is not $\nu$-formal.
\end{lemma}
\begin{proof}
Suppose on contrary that $A$ is $\nu$-formal.
By assumption on $c_i$, the path starting from $i$ and ending at $i+c_i$ in $\vec\DynA_n$ is zero in $A$.
This implies that the $\mathrm{top}(I_{i+c_i})\cong S_{i+1}$, but $c_i<c_{i+1}$ also implies that $\mathrm{soc}(P_{i+1})\neq S_{i+c_i}$, so $I_{i+c_i}$ is non-projective, and there is an exact sequence $P_{i+c_i+1}\to P_{i+1}\to I_{i+c_i}\to 0$.

Since $I_{i+c_i}$ being non-projective says that $\nu_A(I_{i+c_i}) = H^0(\nu I_{i+c_i}) = 0$, applying $\nu_A = -\otimes_ADA$ to the above exact sequence we get that an exact sequence $I_{i+c_i+1}\to I_{i+1}\to 0$.
As $\mathrm{top}(I_{i+c_i+1})=S_{i+1}$, this exact sequence implies that $I_{i+1}= S_{i+1}$, which contradicts the fact that $A$ is ring-indecomposable (i.e. $I_1$ is the unique simple injective module).
\end{proof}

\begin{lemma}\label{lem-non-nf-naka2}
Suppose $A$ is a ring-indecomposable quotient algebra of $K\vec\DynA_n$.
If there exists $1<i<n$ such that $c_{i-1}>c_i\leq c_{i+1}$, then $A$ is not $\nu$-formal.
\end{lemma}
\begin{proof}
As in the proof of Lemma \ref{lem-non-nf-naka1}, the assumption $c_i\leq c_{i+1}$ says that $\mathrm{top} I_{i+c_i}=S_{i+1}$.
On the other hand, the assumption $c_{i-1}>c_i$ says that $\mathrm{top} I_{i+c_i-1}=S_j$ for some $j\leq i-1$.
Therefore, the length of $I_{i+c_i-1}$ is larger than that of $I_{i+c_i}$.
This means that the opposite ring $A^\mathrm{op}$ is a ring-indecomposable quotient algebra of $K\vec\DynA_n$ satisfying the condition (1) of Lemma \ref{lem-naka-3types}, and we can apply Lemma \ref{lem-non-nf-naka1}.
\end{proof}

\begin{lemma}\label{lem-non-nf-naka3}
Consider $A= T_{n,l}$ with $n=lt+r$ for some $0\leq r< l$ and $t\geq 1$.
If $l\neq 2$ and $r\neq 1$, then $A$ is not $\nu$-formal.
\end{lemma}
\begin{proof}
First assume that $r\notin \{0,1\}$, which also forces $l\neq 2$.
Observe that by our assumption, the minimal injective coresolution of $P_{n-r+1}$ is of the form
\[
0\to P_{n-r+1} \to I_n \to I_{n-r} \to I_{n-l} \to I_{n-l-r} \to I_{n-2r}\to \cdots \to I_{l+r}\to I_l\to I_r\to 0.
\]
More precisely, the sequence $(I^k)_{k\geq 0}$ of injective modules in the coresolution is given by
\[
I^k = \begin{cases}
I_{n-lj}, & \text{if }k=2j\leq 2t; \\
I_{n-lj-r}, & \text{if }k=2j+1 < 2t;\\
0, & \text{otherwise.}
\end{cases}
\]
Applying $\nu_A^-$ to this exact sequence, we see that $\nu^{-1}(P_{n-r+1}) \cong M[-2t]$ where $M\cong P_r/\rad^{l-r}P_r$.

Now we have an exact sequence $0\to M\to I_{l-1}\to I_{r-1}\to 0$ given by the minimal injective coresolution of $M$.
Applying $\nu_A^-$ we see that $H^i(\nu^{-1}(M))\neq 0$ for both $i=0$ and $i=1$.
Therefore, $\nu^{-2}(P_{n+r-1})[2t]=\nu^{-1}(M)$ is not a stalk complex and thus, $A$ is not $\nu$-formal.

Assume now that $r=0$ and $l\neq 2$.
We carry out a similar calculation starting with $P_n$.
The minimal injective coresolution of $P_n$ is given by
\[
0 \to P_n \to I_n \to I_{n-1} \to I_{n-l} \to I_{n-l-1} \to I_{n-2l} \to \cdots \to I_l \to I_{l-1} \to 0,
\]
i.e. for all $0\leq j< t$, the $2j$-th term is given by $I_{n-l(j-1)}$ and the $(2j+1)$-st term is given by $I_{n-lj-1}$.
Applying $\nu_A^-$ to this exact sequence, we see that $\nu^{-1}(P_n)\cong S_{l-1}[1-2t]$.

Now we have exact sequence $0\to S_{l-1}\to I_{l-1}\to I_{l-2}\to 0$.
Since $l\neq 2$, $\nu^{-1}(S_{l-1}) = (0\to P_{l-1}\to P_{l-2}\to 0)$ has non-zero cohomology at both degrees 0 and 1.
Hence, $\nu^{-2}(P_n)[2t-1] = \nu^{-1}(S_{l-1})$ is not a stalk complex, which means that $A$ is not $\nu$-formal.
\end{proof}

Finally, we note a recent result of Vaso, which classifies $d$-representation-finite quotients of $K\vec\DynA_n$.
\begin{theorem}{\rm \cite{Vas}}\label{thm-vaso}
A quotient algebra $A$ of $K\vec\DynA_n$ is $d$-representation-finite if, and only if, $A\cong T_{n,l}$ with $l=2$ or $l$ divides $n-1$; in which case, we have $d=2(n-1)/l$ respectively.
\end{theorem}

The following classification result now follows easily.
\begin{theorem}\label{thm-nu-formal-naka}
Suppose $A$ is a ring-indecomposable quotient algebra of $K\vec\DynA_n$.
Then the following are equivalent.
\begin{enumerate}[(1)]
\item $A$ is $\nu$-formal.
\item $A$ is periodically $\nu$-formal.
\item $A$ is $d$-hereditary for some $d\geq 1$.
\item $A$ is $d$-representation-finite for some $d\geq 1$.
\item $A\cong T_{n,l} = K\vec\DynA_n / \rad^l(K\vec\DynA_n)$ with $l=2$ or $l$ divides $n-1$.
\end{enumerate}
\end{theorem}
\begin{proof}
The implication form (4) to (2) (hence, to (1)) is trivial.
Theorem \ref{thm-vaso} gives the equivalence between (3), (4), and (5).
The equivalence between (1) and (5) follows from combining Lemma \ref{lem-naka-3types} with Lemma \ref{lem-non-nf-naka1}, \ref{lem-non-nf-naka2}, and \ref{lem-non-nf-naka3}.
\end{proof}

\subsection{Further example II: Higher canonical algebra}\label{subsec-NF-canonical}
We give another class of $\nu$-formal algebras with finite global dimension, but not necessarily $d$-hereditary.
This class of algebra is called \emph{$d$-canonical algebra} (for some integer $d\geq 1$), and is first studied in \cite{HIMO} as a generalisation of the classical canonical algebras \cite{Rin,GeiLen}.

To keep exposition concise, we will give minimal detail about many of the results we use, but refer the reader to \cite{HIMO}.

Fix positive integers $n\geq 1$ and $d\geq -1$.
Let $R$ be a so-called Geigle-Lenzing complete intersection ring of Krull dimension $d+1$ with weights $p_1, p_2, \ldots, p_n$, where
each $p_i$ is a positive integer with $p_i\geq 2$.
Then there is an abelian category $\mathrm{coh}\X$ given by the coherent sheaves on the associated Geigle-Lenzing projective space $\X$.
Denote by $\Ocal$ the structure sheaf of $\X$.

Define an abelian group $\Ll$ generated by $\vec x_1, \vec x_2, \ldots, \vec x_n, \vec c$ with relations $p_i\vec x_i - \vec c$ for all $1\leq i\leq n$.
This group acts on $\mathrm{coh}\X$ naturally, i.e. there are auto-equivalences $?(\vec x)$ on $\mathrm{coh}\X$ for any $\vec x\in \Ll$ with inverse $?(-\vec x)$.
There is a natural poset structure on the abelian group $\Ll$ generated by the underlying monoid, i.e. $\vec x \geq \vec y$ if $\vec x-\vec y$ is an element of the submonoid generated by $\vec x_i$'s and $\vec c$.
We use the usual notation $[\vec x,\vec y]$ to denote the interval $\{\vec z\mid \vec x\leq \vec z\leq \vec y\}$ under this poset structure.

By \cite[Theorem 6.1]{HIMO}, the \emph{$d$-canonical algebra} of $\X$ is the endomorphism ring $\End_\X(T)$ of a tilting object $T=\bigoplus_{\vec x\in [0, d\vec c]} \Ocal(\vec x)$ in $\D$.
For $d=1$, this is precisely the canonical algebra of Ringel \cite{GeiLen}.
The following is the main result of this subsection.

\begin{theorem}\label{thm-canonical}
Any $d$-canonical algebra is $\nu$-formal.
\end{theorem}

Define an element $\vec\omega$ of $\Ll$ by
\[
\vec\omega := (n-d-1) \vec c - \sum_{i=1}^n \vec x_i.
\]

By \cite[Theorem 5.4]{HIMO}, $\mathrm{coh}\X$ has global dimension $d$ and the Serre functor $\nu$ on the bounded derived category $\D:=\Db(\mathrm{coh} \X)$ is given by the composition of twist by $\vec \omega$ and shift by $d$, i.e. $\nu\cong (\vec\omega)[d]$.

Since there is an equivalence of triangulated categories $\Db(\mod A)\simeq\D$ given by $\RHom_\X(T,-)$, for the purpose of proving Theorem \ref{thm-canonical}, it suffices to understand various vanishing conditions of the Hom-space $\Hom_\D(\Ocal, \Ocal(\vec x)[l]) = \Ext_\X^l(\Ocal,\Ocal(\vec x))$ for $\vec x\in \Ll$ and $l\in \Z_{\geq 0}$.
The following result will be useful.

\begin{proposition}\label{prop-canonical}{\rm \cite[Observation 3.1(c), Proposition 5.3]{HIMO}}
\begin{enumerate}
\item $\Ext_\X^i(\Ocal(\vec x),\Ocal(\vec y))\neq 0$ implies that $i=0$ or $d$.
\item $\Hom_\X(\Ocal,\Ocal(\vec x))\neq 0$ if, and only if, $\vec x\geq 0$.
\item $\Ext_\X^d(\Ocal,\Ocal(\vec x))\neq 0$ if, and only if, $\vec\omega-\vec x\geq 0$.
\end{enumerate}
\end{proposition}

\begin{proof}[Proof of Theorem \ref{thm-canonical}]
By construction, the equivalence $\D \simeq \Db(\mod A)$ given by $\RHom_\X(T,-)$ sends $\Ocal(\vec x)$ to the indecomposable projective $A$-modules $P_{\vec x}$ for all $\vec x\in [0,d\vec c]$.
Hence, for any $k,p\in \Z$ and $\vec x\in [0,d\vec c]$, we have an isomorphism $H^p(\nu^k P_{\vec x})\cong \Hom_\D(T,\nu^k\Ocal(\vec x)[p])$.
Therefore, we just need to show the this space is non-zero for at most one $p\in \Z$.

Note that we have isomorphisms
\[
\Hom_\D(T,\nu^k\Ocal(\vec x)[p]) \cong \Hom_\D(T,\Ocal(\vec x +k\vec\omega)[dk+p]) \cong \Ext_\X^{dk+p}(T,\Ocal(\vec x+k\vec\omega)).
\]
It follows from Proposition \ref{prop-canonical} (1) that this space is non-zero for at most two $p$'s; namely, when $p=-dk$ or $-d(k-1)$.

In the case when the space is non-zero at $p=-dk$, by Proposition \ref{prop-canonical} (2), we have $\vec x+k\vec\omega \geq \vec y \geq 0$ for some $\vec y\in [0,d\vec c]$.
On the other hand, when the space is non-zero at $p=-d(k-1)$, then we have the following isomorphism given by the Serre duality:
\[
\Hom_\D(T,\Ocal(\vec x+k\vec\omega)[d]) \cong D\Hom_\D(\Ocal(\vec x+k\vec\omega)[d],T(\vec\omega)[d]).
\]
Hence, we get that $\Hom_\D(\Ocal(\vec x+k\vec\omega),T(\vec\omega)) \cong \Hom_\X(\Ocal(\vec x+(k-1)\vec\omega), T)$ is also non-zero.
By Proposition \ref{prop-canonical} (3), 
we get that $\vec x+k\vec \omega\leq \vec z+\vec \omega$ for some $\vec z\in [0,d\vec c]$.
Thus, $\vec x+k\vec\omega \leq d\vec c+\vec\omega$.

Therefore, if $H^p(\nu^kP_{\vec x})$ is non-zero at more than one $p$, then we get that 
\[
(n-1)\vec c-\sum_{i=1}^n\vec x_i=d\vec c+\vec\omega \geq \vec x+k\vec\omega\geq 0,
\]
which is a contradiction (see \cite[Observation 3.1(c)]{HIMO}).
\end{proof}

It is not difficult to say slightly more than just $\nu$-formality for $d$-canonical algebras.

\begin{proposition}
A $d$-canonical algebra is periodically $\nu$-formal if, and only if, $\X$ is Calabi-Yau, i.e. $\vec\omega$ is torsion in $\Ll$.
\end{proposition}
\begin{proof}
By Proposition \ref{prop-twistedCY}, we want to show that a $d$-canonical algebra $A$ is twisted fractionally Calabi-Yau precisely when $\X$ is Calabi-Yau.
From the equivalence $\RHom_\X(T,-):\D\xrightarrow{\sim}\Db(\mod A)$, we get that $A$ is twisted fractionally Calabi-Yau if, and only if, there is some non-zero $k\in \Z$ so that $\nu^k(T)$ is isomorphic to a shift of $T$.
Since there is a natural isomorphism $\nu\simeq (\vec\omega)[d]$, we have 
\[
\nu^k(T)\cong \bigoplus_{\vec x\in [0,d\vec c]}\Ocal(\vec x+k\vec\omega)[dk] \cong \bigoplus_{\vec x\in [k\vec\omega, d\vec c+k\vec\omega]}\Ocal(\vec x)[dk]
\]
for all $k\in \Z$.
Therefore, $\nu^k(T)$ is isomorphic to a shift of $T$ if, and only if, $k\vec\omega=0$, i.e. $\vec\omega$ is torsion.
\end{proof}

\section{SGC extension algebras}\label{sec-SGC}

In \cite[Example 2]{Yam}, Yamagata considered an iterative construction of algebras, starting from hereditary algebras.
We call such a construction the \emph{SGC extensions} of the original algebra, where SGC stands for \underline{s}mallest \underline{g}enerator-\underline{c}ogenerator.
He observed that the global and dominant dimensions of the algebras in such a series is always finite (without detailed proof) - in fact, this inspires all the investigations presented in this article, and this section in particular.

We will first show a connection between SGC extension algebras and replicated algebras when the original algebra satisfies certain Hom-vanishing condition in subsection \ref{subsec-SGC-replicate}.
In particular, the two constructions almost always coincide when the base algebra is hereditary, which allows us to recover, and improve, the observation of Yamagata by utilising the results from the previous sections (namely, Theorem \ref{thm-dims}).
In subsection \ref{subsec-SGC-typeA}, we give explicit calculation on the homological dimensions of SGC extensions of certain algebras.
These calculations allow us to cover some of the ``bad" cases where Theorem \ref{thm-dims} and Theorem \ref{thm-dims2} fail to generalise to SGC extensions in a straightforward manner.
Finally, we give some further remarks about SGC extensions in subsection \ref{subsec-SGC-remarks}.

Let us start by defining SGC extensions properly.

\begin{definition}
Let $A$ be a finite dimensional $K$-algebra.
The \emph{SGC extension} of $A$ is the endomorphism ring $\End_A(A\oplus DA)$.
More generally, by taking $A^{[0]}:=A$, we define a series of algebras $\{A^{[m]}\}_{m\geq 0}$ iteratively by $A^{[m+1]}:=\End_{A^{[m]}}(A^{[m]}\oplus DA^{[m]})$.
The algebra $A^{[m]}$ is called the \emph{$m$-th SGC extension} of $A$.
For simplicity, the basic algebra of $A^{[m]}$ is  called the \emph{$m$-th basic SGC extension} of $A$.
\end{definition}

\subsection{Relation to replicated algebras}\label{subsec-SGC-replicate}
The following result shows that, when the projective and injective modules satisfy a specific condition, then the $m$-th basic SGC extension algebra is isomorphic to an idempotent truncation of the $m$-th replicated algebra.

\begin{theorem}\label{thm-SGC-replicate}
Let $A=A^{[0]}$ be a finite dimensional algebra and $e$ be an idempotent of $A$ so that $\add A \cap \add DA = \add (1-e)A$.
If $\Hom_A(DA,eA)=0$, then the $m$-th SGC extension $A^{[m]}$ of $A$ is Morita equivalent to an idempotent truncation 
\[
e^{[m]}A^{(m)}e^{[m]} =
\left(\begin{array}{cccccc}
eAe & 0 & 0 & \cdots & 0 & 0 \\
e(DA)e & eAe & 0 &  & & 0 \\
0 & \ddots & \ddots &  & & \vdots \\
\vdots &  & \ddots & \ddots & &\vdots \\
0 & & & e(DA)e & eAe &0 \\
0 & 0 & \cdots & 0 & (DA)e & A
\end{array}\right).
\]
of the $m$-replicated algebra $A^{(m)}$, where the idempotent $e^{[m]}$ is given by an $(m+1)\times(m+1)$ diagonal matrix with entries $(e,e,\ldots, e,1)$.
In particular, $A^{[m]}$ is Morita equivalent to 
\[
\left(\begin{array}{cc}(eAe)^{(m)} & 0 \\ \ [(1-e)Ae]_m & (1-e)A(1-e)\end{array}\right).
\]
\end{theorem}
\begin{proof}
We prove the theorem by induction on $m$.
For $m=1$, by definition, $A^{[1]}$ is Morita equivalent to $\End_A(eA\oplus DA)$, which can be written as
\begin{align}
\left(\begin{array}{cc}\Hom_A(eA,eA) & \Hom_A(DA,eA) \\ \Hom_A(eA,DA) & \Hom_A(DA,DA)\end{array}\right)
& \cong \left(\begin{array}{cc}eAe & 0 \\ (DA)e & A\end{array}\right) \notag \\
& \cong \left(\begin{array}{cc}e & 0 \\ 0 & 1 \end{array}\right) \left(\begin{array}{cc}A & 0 \\ DA & A \end{array}\right) \left(\begin{array}{cc}e & 0 \\ 0 & 1 \end{array}\right) \cong e^{[1]} A^{(1)} e^{[1]}.\notag 
\end{align}

Suppose the hypothesis is true for some $m\geq 1$, i.e. $A^{[m]}$ is Morita equivalent to the basic algebra $B:=e^{[m]} A^{(m)} e^{[m]}$.
Using the matrix form of $B$, it can be seen that the projective $B$-modules $[(DA)e,A]_{m-1}$ and $[e(DA)e,eAe]_i\cong [D(eAe),eAe]_i$ for $0\leq i\leq m-2$ are also injective.
Moreover, the projective $B$-module $[(DA)e, A]_{m-1}$ is isomorphic to $[e(DA)e,eA]_{m-1}\oplus [(1-e)(DA)e,(1-e)A]_{m-1}$.
Since $(1-e)(DA)e=0$ and $(1-e)A\in \add DA$ by our assumption, we have $[(DA)e,A]_{m-1}$ being injective.
Thus, $[eAe]_0\oplus DB$ is the smallest generator-cogenerator and $A^{[m+1]}$ is Morita equivalent to $\End_{B}([eAe]_0\oplus DB)$, which can be written as
\[ \left(\begin{array}{cc}
\Hom_{B}\Big([eAe]_0,[eAe]_0\Big) & \Hom_{B}\Big(DB,[eAe]_0\Big) \\
\Hom_{B}\Big([eAe]_0,DB\Big) & \Hom_{B}\Big(DB, DB\Big)
\end{array}\right) \cong  \left(\begin{array}{cc}
\Hom_{eAe}\Big(eAe,eAe\Big) & \Hom_{B}\Big(DB,[eAe]_0\Big) \\
\Hom_{B}\Big([eAe]_0,DB\Big) & B
\end{array}\right).
\]

Note that $DB = \left(\bigoplus_{i=0}^{m-2} [e(DA)e, eAe]_i\right)\oplus [e(DA)e,eA]_{m-1}\oplus [DA]_{m}$, so all direct summands of $DB$ apart from $[e(DA)e,eAe]_0$ has $0$-th entry being zero,
which means that the upper-right entry is
\[
\Hom_{B}\Big(DB,[eAe]_0\Big) \cong \Hom_{B}\Big([e(DA)e,eAe]_0,[eAe]_0\Big)  
\]
By the description of maps between modules we gave in Section \ref{sec-Repl-Alg}, 
any map in this space is given by $(f,0):(e(DA)e,eAe)\to (eAe,0)$ with $f\circ \mathrm{id}=0$; hence, this space is always zero.

For the bottom-left entry, we have isomorphisms
\[
\Hom_{B}\Big([eAe]_0,DB\Big) \cong \Hom_{K}\Big([eAe]_0, K\Big) \cong [D(eAe)]_0^{\top} \cong [e(DA)e]_0^{\top},
\]
where $[X]_0^\top$ denotes the transpose of the row vector $[X]_0$.
Hence, $\End_{B}([eAe]_0\oplus DB)$ is isomorphic to $\left(\begin{array}{cc} eAe & 0 \\ \ [e(DA)e]_0^\top & B \end{array}\right)$.
If we write $B=e^{[m]}A^{(m)}e^{[m]}$ in its matrix form, then $e(DA)e$ in the bottom-left entry is regarded as a column vector with $e(DA)e$ in the top entry and everywhere else zero.
It is now clear that this is precisely $e^{[m+1]}A^{(m+1)}e^{[m+1]}$ as required.

For the final statement, we show the stated matrix algebra is isomorphic to $e^{[m]}A^{(m)}e^{[m]}$.
By our assumption, we have $\Hom_A((1-e)A,eA)\cong eA(1-e)=0$.
This means that we have 
\[
(DA)e = \left( \begin{array}{c} e(DA)e \\ 0 \end{array}\right) \text{ and }A\cong \left( \begin{array}{cc}eAe&0 \\ (1-e)Ae & (1-e)A(1-e)A\end{array}\right)\]
So we can replace the bottom row of $e^{[m]}A^{(m)}e^{[m]}$ by a $2\times (m+2)$-matrix 
\[
\left(\begin{array}{cccccc}
0 & \cdots & 0 & e(DA)e & eAe     & 0\\
0 & \cdots & 0 & 0      & (1-e)Ae & (1-e)A(1-e)
\end{array}.\right)
\]
Thus, we obtain a $(m+2)\times (m+2)$-matrix with the top-left $(m+1)\times(m+1)$-submatrix being $(eAe)^{(m)}$.
This completes the proof.
\end{proof}

Note that if $A$ is $d$-hereditary or $d$-canonical, then the condition of Theorem \ref{thm-SGC-replicate} is satisfied.
If $A$ is self-injective, then Theorem \ref{thm-SGC-replicate} simply says that $A^{[m]}$, which is Morita equivalent to $A$, is the ``embedded in the corner entry" of $A^{(m)}$.

The following follows immediately from Theorem \ref{thm-SGC-replicate} by taking $e=1$.
\begin{corollary}\label{cor-e=1}
Let $A$ be a $\nu$-formal algebra with $\Hom_A(DA,A)=0$.
Then $A^{(m)}$ and $A^{[m]}$ are Morita equivalent for all $m\geq 1$.
\end{corollary}

\begin{example}\label{eg-naka-SGC} 
Consider the Nakayama algebra $A=T_{n,l+1}=K\vec\DynA_n/\rad^{l+1}(K\vec\DynA_n)$.
Then $A$ satisfies the condition of Theorem \ref{thm-SGC-replicate} with $e=e_{n-l+1}+e_{n-l+2}+\cdots +e_n$, where $e_i$ is the primitive idempotent corresponding to the vertex $i\in \vec\DynA_n$.
Then $A$ can be presented as a matrix algebra with $(i,j)$-entry $K$ for all $1\leq i\leq n$ and $i-l\leq j\leq i$; zero otherwise.
Then $eAe$ is given by the top-left $l\times l$-matrix with $(1-e)Ae$ being a $(n-l)\times l$-matrix and $(1-e)A(1-e)$ being the bottom-right $(n-l)\times (n-l)$-matrix.
Note that $eAe\cong K\vec \DynA_{l}$ and $(1-e)A(1-e)\cong K\vec \DynA_{n-l}$.
In particular, $(eAe)^{(m)}\cong T_{l(m+1),l+1}$ as we have mentioned in Example \ref{eg-typeA-Kronecker} (1).

It now follows from the theorem and the matrix presentation of $T_{l(m+1),l+1}$ that $A^{[m]}$ can be presented by a square matrix of size $l(m+1)+(n-l)=n+ml$ with $(i,j)$-entry $K$ for all $1\leq i\leq n+ml$ and $i-l\leq j\leq i$; zero otherwise.
In other words, the basic algebra of $T_{n,l+1}^{[m]}$ is $T_{n+ml,l+1}$.
\end{example}

We list two easy consequences of Corollary \ref{cor-minAG2}.
\begin{corollary}\label{cor-SGC-replicate}
The following are equivalent for a $\nu$-formal algebra $A$ satisfying $\Hom_A(DA,A)=0$.
\begin{enumerate}[(1)]
\item $A$ is periodically $\nu$-formal.
\item There exists some $m\in \Z_+$ such that $A^{[m]}$ is minimal Auslander-Gorenstein.
\item There exists infinitely many $m\in \Z_+$ such that $A^{[m]}$ is minimal Auslander-Gorenstein.
\end{enumerate}
\end{corollary}

In \cite{Iya1} (resp. \cite{IS}), it is shown that any $n$-Auslander algebra (resp. $n$-Auslander-Gorenstein algebra) $\Gamma$ is isomorphic to the endomorphism ring of a so-called \emph{$n$-cluster-tilting} (resp. \emph{$n$-precluster-tilting}) $\Lambda$-module $M$.
Moreover, the pair $(\Lambda,M)$ is unique up to Morita equivalence of $\Lambda$ and a choice of generator of $\add M$.
We call $(\Lambda,M)$ the \emph{higher (pre)cluster-tilting pair corresponding to $A$}.

\begin{corollary}\label{cor-SGC-higherAus}
Suppose $A$ is a periodically $\nu$-formal algebra with $\Hom_A(DA,A)=0$ and has finite (resp. infinite) global dimension.
If the $(m+1)$-st replicated algebra $A^{(m+1)}$ is a higher Auslander (resp. minimal Auslander-Gorenstein) algebra, then the  corresponding higher (pre)cluster-tilting pair is $(A^{[m]}, A^{[m]}\oplus DA^{[m]})$.
\end{corollary}
\begin{proof}
Under the assumption of the claim, $A^{(m+1)}$ is Morita equivalent to $A^{[m+1]}$ by Corollary \ref{cor-e=1}; hence, the later is higher Auslander (resp. minimal Auslander-Gorenstein).
The claim now follows from the higher Auslander correspondence (resp. Iyama-Solberg correspondence).
\end{proof}

When the condition $\Hom_A(DA,A)=0$ is replaced by the one in Theorem \ref{thm-SGC-replicate}, it is natural to ask whether we can compare the homological dimensions of the series $A^{[m]}$ and that of the series $A^{(m)}$; in particular, whether Corollary \ref{cor-SGC-replicate} still holds.
While a general method is desirable, we do not have much success.
We can, however, rescue the result in the case of hereditary algebras as follows.

\subsection{SGC extensions of certain Nakayama algebras}\label{subsec-SGC-typeA}
Theorem \ref{thm-SGC-replicate} implies that the SGC extensions of almost all hereditary algebras coincide with their replicated algebras up to Morita equivalence.
The only family of exceptions is given by hereditary Nakayama algebra (i.e. whose valued quiver is $\vec\DynA_n$).

In this subsection, we determine the formulae for the dominant and global dimensions of the SGC extensions of the Nakayama algebras with $n$ simple modules and Kupisch series $[l,l,\ldots, l,l-1,\ldots, 2, 1]$.
In particular, this gives a classification of higher Auslander algebras taking such a form.
Since the computation of the homological dimensions for Nakayama algebras depends only on the Kupisch series, we can assume that the algebras are given by $T_{n,l}:=K\vec \DynA_n/\rad^l(K\vec \DynA_n)$.

We will also give some ring theoretic observations about these SGC extensions in comparison to Proposition \ref{prop-QF}.

\medskip

For convenience, we will use $T_{n,l}^{[m]}$ to denote the basic algebra of the $m$-th SGC extension throughout, instead of the original definition.
Since $T_{n,l}^{[m]}\cong T_{n+m(l-1),l}$ as we have shown in Example \ref{eg-naka-SGC}, we can just calculate $\domdim T_{n,l}$ and $\idim T_{n,l}$ for arbitrary $n,l$ to achieve our goal.
Note that the case of $l=2$ is already done in \cite[Example 2.4(a)]{Iya2}, but our strategy below works for the case $l=2$, too.

It turns out that there are recursive formulae determining these dimensions.
It is more convenient for us to work over an infinite dimensional algebra $A$, where $T_{n,l}$ appears as a quotient, as well as an idempotent truncation, of $A$, for all $n\geq 1$.

Let $Q$ be the linear oriented quiver of type $\DynA_\infty$, i.e.
\[
Q: \qquad 1 \to 2 \to 3 \to \cdots 
\]
The infinite dimensional algebra $A$ is defined as $KQ/\rad^l(KQ)$.
As usual, we denote by $e_i$ the primitive idempotent given by the trivial path at the vertex $i$, and $P_i, I_i, S_i$ the corresponding indecomposable projective, indecomposable injective, simple $A$-modules respectively.
The finite dimensional indecomposable $A$-modules are of the form $M_{i,s}:=e_iA/e_i\rad^s(A)$ for some $i \geq 1$ and $s$ with $1 \leq s \leq l$.
Note that every $P_i$ is injective and isomorphic to $I_{i+l-1}$, but $I_i$ is projective only for $i\geq l$.
It is easy to see that the algebra $T_{n,l}$ is isomorphic to $eAe\cong A/A(1-e)A$ where $e=e_1+e_2+\cdots +e_n$; roughly speaking, $T_{n,l}$ is obtained from $A$ by deleting the vertices $i\in Q_0$ for all $i > n$.

Define $d_{i,s}:=\domdim_A(M_{i,s})$ and $g_{i,s}:=\idim_A(M_{i,s})$ for all $i\geq 1$ and $1 \leq s \leq l$.
For convenience, we set $d_{i,s}=0=g_{i,s}$ for any pair of integers $(i,s)$ with $i \leq 0$.

\begin{lemma}\label{lem-recursion}
Fix some $i\geq 1$.
If $1\leq s\leq l-i$, then we have 
\[
d_{i,s}=0 \quad \text{ and }\quad g_{i,s}=\begin{cases}
0, & \text{if }i=1;\\ 1, & \text{if }i>1.
\end{cases}
\]
Otherwise, we have the following recursions
\[
d_{i,s}=1+d_{i+s-l,l-s} \quad \text{ and }\quad g_{i,s}=1+g_{i+s-l,l-s}.
\]
In particular, every indecomposable non-injective $A$-module $M$ satisfies $\idim_A M-\domdim_A M\leq 1$.
\end{lemma}
\begin{proof}
Recall that $M_{1,s}\cong I_s$ for $1\leq s\leq l-1$ are all the indecomposable injective non-projective $A$-modules, which means that its dominant and injective dimensions (i.e. $d_{1,s}$ and $g_{1,s}$) are both zero.
Since $I_i$ is projective is equivalent to $i\geq l$, the first term in the injective coresolution $M_{i,s}$ is projective if, and only if, $i+s-1\geq l$.

Suppose we have $s\leq l-i$.
We have already explained the case when $i=1$; for the case of $i>1$, we have $\Omega^{-1}(M_{i,s}) \cong I_{i-1}$, which yields $g_{i,s}=1$ and $d_{i,s}=0$.

In the case of $s>l-i$, we get that the injective envelop of $M_{i,s}$ is projective, and the cosyzygy $\Omega^{-1}(M_{i,s})$ is then isomorphic to $M_{i+s-l,l-s}$.
The two recursion formulae now follow.

For the last statement, observe that it holds for $M_{i,s}$ with $s<l-i$, then apply the recursions.
\end{proof}

We have the following formulae; note that $\gldim T_{n,l}$ is also calculated in \cite{Vas}.
\begin{proposition} \label{prop-naka-1}
Let $n\geq l$ be a positive integers and write $n=lt+r$ for some $0\leq  r< l$ and $t\geq 1$.
Then we have
\begin{align}
\gldim T_{n,l}=\idim  T_{n,l} = \begin{cases}
2t-1 & \text{if }r=0;\\
2t   & \text{if }r=1;\\
2t+1 & \text{otherwise},
\end{cases} \quad \text{and}\quad 
\domdim  T_{n,l} = \begin{cases}
2t-1 & \text{if }r<l-1;\\
2t & \text{if }r=l-1.
\end{cases}\notag 
\end{align}
In particular, $T_{n,l}$ is a higher Auslander algebra if, and only if, $l=2$ or $l$ divides $n$.
Moreover, in such a case the higher Auslander corresponding pair is $(T_{n-l+1,l},T_{n-l+1,l}\oplus D(T_{n-l+1,l}))$.
\end{proposition}
\begin{proof}
By construction, the indecomposable projective non-injective $T_{n,l}$-modules are given by $e_i T_{n,l}$ with $n-l+1< i\leq n$, which is isomorphic to $M_{i,n-l+1-i}$ as $A$-module, which means that
\begin{align}
\idim(T_{n,l})_{T_{n,l}} & = \max \{g_{n-i,i+1} \mid 0\leq i<l-1 \},\notag \\
\text{ and } \quad \domdim(T_{n,l})_{T_{n,l}} &=  \min \{d_{n-i,i+1} \mid 0\leq i< l-1\}. \notag 
\end{align}

It follows from Lemma \ref{lem-recursion} that $g_{n-i,i+1} \leq g_{n-j,j+1}$ and $d_{n-i,i+1}\leq d_{n-j,j+1}$ whenever $i \leq j$ for $i,j<l-1$.
In particular, it follows from the discussion above that we have $\domdim T_{n,k}=d_{n-l+2,l-1}$ and $\idim T_{n,k}=g_{n,1}$.

By the recursion in Lemma \ref{lem-recursion}, we get that
\[
g_{n,1} = g_{lt+r,1} = 1+g_{l(t-1)+r+1,l-1} = 2+g_{l(t-1)+r,1} = \cdots = 2(t-1)+1+g_{r+1,l-1}.
\]
If $r=0$, then we get that $g_{r+1,l-1}=g_{1,l-1}=0$ and $\idim T_{n,l} = g_{n,1} = 2(t-1)+1=2t-1$.
If $r\neq 0$, we get that 
\[
\idim T_{n,l} = 2(t-1)+2+g_{r,1} = \begin{cases} 2t, &\text{if }r=1 \\ 2t+1, &\text{if }r>1.
\end{cases}
\]

Similarly, for the dominant dimension, we get that 
\begin{align}
\domdim T_{n,l}&= d_{n-l+2,l-1} = d_{l(t-1)+r+2,l-1} = 1+d_{l(t-1)+r+1,1} = 2+d_{l(t-2)+r+2,l-1} = \cdots \notag \\
&= 2(t-1)+1+d_{r+1,1} = \begin{cases}
2t-1, & \text{if }r<l-1;\\
2t & \text{if }r=l-1.
\end{cases}\notag 
\end{align}

The criteria of $T_{n,l}$ being higher Auslander follows immediately from these formulae.
Moreover, in such a case, it follows from Example \ref{eg-naka-SGC} (which says that $T_{n-l+1,l}^{[1]}\cong T_{n,l}$) that the higher Auslander corresponding pair is precisely as stated in our claim.
\end{proof}
Note that Proposition \ref{prop-naka-1} says that $T_{n,l}$ is a higher hereditary algebra when $l$ divides $n-1$ or $l=2$.
Hence, combining Proposition \ref{prop-naka-1} with Lemmas \ref{lem-non-nf-naka1}, \ref{lem-non-nf-naka2}, \ref{lem-non-nf-naka3} gives an alternative proof for Vaso's result (Theorem \ref{thm-vaso}).

\begin{theorem} \label{thm-naka-SGC}
Let $m$ be a positive integer.
\begin{enumerate}[(1)]
\item $T_{n,l}^{[m]}$ is a higher Auslander algebra if, and only if, $l=2$ or $l$ divides $|n-m|$.
\item $\idim M - \domdim M\leq 1$ for all $M\in \mod T_{n,l}^{[m]}$.
\item $T_{n,l}^{[m]}$ is a QF-13 algebra.
\end{enumerate}
\end{theorem}
\begin{proof}
(1): For $l>2$, apply Proposition \ref{prop-naka-1} to $T_{n,l}^{[m]}\cong T_{n+m(l-1),l}$.

For $l=2$, since $T_{n,l}^{[m]}\cong T_{n+m,2}$, the result is already shown in \cite[Example 2.4(a)]{Iya2}.

(2): Immediate from Lemma \ref{lem-recursion}.

(3): 
By Proposition \ref{prop-yamagata-QF13}, we need to prove that every indecomposable $T_{n+m(l-1),l}$-module $M_{i,s}$ has $\domdim M_{i,s}\geq 1$ or $\codomdim M_{i,s}\geq 1$.
Note that the dominant dimension $d_{i,s}$ of $M_{i,s}$ is zero precisely when $i<l$ and $1\leq s\leq l-i$ (cf. Lemma \ref{lem-recursion}).
Clearly the projective cover (both as $A$- and $T_{n+m(l-1),l}$-module) $P_i$ of such $M_{i,s}$ is injective, which means that $\codomdim M_{i,s}\geq 1$ as required.
\end{proof}

Now, combining Corollary \ref{cor-SGC-replicate} with Theorem \ref{thm-naka-SGC}, we have the following relation between SGC extensions and representation-finiteness of hereditary algebras.
\begin{corollary}\label{cor-Yam-ext-high-Aus}
The following are equivalent for a hereditary algebra $A$.
\begin{enumerate}[(1)]
\item $A$ is representation-finite.
\item There exists some $m\in \Z_+$ such that $A^{[m]}$ is a higher Auslander algebra.
\item There exists infinitely many $m\in \Z_+$ such that $A^{[m]}$ is a higher Auslander algebra.
\end{enumerate}
\end{corollary}

\subsection{Further remarks}\label{subsec-SGC-remarks}

Recall from Proposition \ref{prop-repli-tensor-SI} that taking the $m$-replicated algebra is an operation which ``commutes" with taking the tensor product algebra with a self-injective algebra.
The analogous statement is also true for $m$-th SGC extensions.

\begin{proposition}\label{prop-tensor2}
Let $A$ be an algebra with $m$-th SGC extension $A^{[m]}$ and $B$ a basic and self-injective algebra.
The $m$-th SGC extension of $A \otimes_K B$ is Morita equivalent to $A^{[m]} \otimes_K B$.
\end{proposition}
\begin{proof}
For simplicity, we omit the subscript $K$ on $\otimes_K$.
Note that $D(A \otimes B) \cong DA \otimes DB$ and $DB \cong B$ since $B$ is assumed to be basic and self-injective. It suffices to prove the proposition for $m=1$: 
\begin{align*}
(A \otimes B)^{[1]} &=\End_{A \otimes B}\big((A \otimes B) \oplus D(A \otimes B)\big) \cong \End_{A \otimes B}\big( (A \otimes B) \oplus (DA \otimes B)\big) \\
& \cong \End_{A \otimes B}\big( (A \oplus DA) \otimes B\big) \cong \End_A(A \oplus DA) \otimes \End_B(B) \cong A^{[1]} \otimes B. \qedhere
\end{align*}
\end{proof}

This gives us yet another new construction of minimal Auslander-Gorenstein algebras.

\begin{example}
Let $A=K\vec\DynA_2$, $B=K[x]/(x^2)$, and $C=A \otimes_K B$.
Then $C$ is a periodically $\nu$-formal algebra with
\begin{align*}
C^{[m]} \cong A^{[m]}\otimes_K B \cong T_{m+2,2}\otimes_KB,\;\;\text{ and }\;\; C^{(m)} \cong A^{(m)}\otimes_K B \cong T_{2m+2,3}\otimes_KB.
\end{align*}
By our results, $C^{[m]}$ (resp. $C^{(m)}$) is minimal Auslander-Gorenstein for all even $m$ (resp. $m=3t-1$ for some positive integer $t$).
One can show that $C^{[m]}$ is representation-finite and $C^{(m)}$ is representation-infinite for all $m \geq 1$.
\end{example}

The next example shows that SGC extensions need not preserve representation-type, Iwanaga-Gorenstein property, as well as non-Iwanaga-Gorenstein property.

\begin{example}
Let $K$ be an algebraically closed field.
Define $A:=KQ/I$ where
\[
Q:= \xymatrix{ & 1\ar@(ul,dl)_{a} \ar[r]^{b}   & 2, }
\]
and $I$ is generated by $a^2, ab=0$.
Note that there is no projective-injective $A$-module but $\Hom_A(DA,A)$ is non-zero.

We write the associated data in the following table, where $|\mathrm{ind}\Lambda|$ means the number of isomorphism classes of indecomposable $\Lambda$-modules.

\begin{center}
\begin{tabular}{ccc}
Algebra $\Lambda$ & $|\mathrm{ind}\Lambda|$ & $\idim \Lambda, \pdim D\Lambda$ \\ \hline
$A$       & 5        & $\infty$ \\
$A^{[1]}$ & 24       &  3       \\
$A^{[2]}$ & $\infty$ & $\infty$ 
\end{tabular}
\end{center}

We will not detail the calculations, which can be obtained by manual calculations (or more realistically using a computer, such as the QPA package).
To see that $A^{[2]}$ is representation-infinite, one can calculate that $\dim_K A^{[2]}=28$ and there is a simple $A^{[2]}$-module $S$ with $\dim_K\tau^{10}(S)=59$.
Therefore, there is an indecomposable $A^{[2]}$-module $M$ so that $\dim_K M>\max\{30,2\dim_K A^{[2]}\}$.
It follows by a critera of Bongartz \cite[3.9]{Bo} that $A^{[2]}$ is representation-infinite.
\end{example}


\begin{thebibliography}{Gus}
\bibitem[ABST]{ABST}
{\sc I. Assem, T. Br\"{u}stle, R. Schiffler, G. Todorov},
`{{$m$}-cluster categories and {$m$}-replicated algebras}',
{\it J. Pure Appl. Algebra} {212} (2008), {884--901}.

\bibitem[AssIwa]{AssIwa}
{\sc I. Assem, Y. Iwanaga},
`{On a class of representation-finite QF-3 algebras}',
{\it Tsukuba J. Math.} {11} (1987), no. 2, {199--217}.

\bibitem[AR]{AR}
{\sc M. Auslander, I. Reiten},
`{{$k$}-{G}orenstein algebras and syzygy modules}',
{\it J. Pure Appl. Algebra} {92} (1994), no. 1, {1--27}.

\bibitem[ARS]{ARS}
{\sc M. Auslander, I. Reiten, S. Smalo},
{\it Representation Theory of Artin Algebras.}
Cambridge Studies in Advanced Mathematics, Volume 36, Cambridge University Press, 1997.

\bibitem[Bo]{Bo}
{\sc K. Bongartz}, 
`{Indecomposables live in all smaller length}',
{\it Represent. Theory} 17 (2013), 199--225.

\bibitem[EncJen]{EncJen}
{\sc E. Enochs, O. Jenda},
{\it Relative homological algebra. {V}olume 1}, 
De Gruyter Expositions in Mathematics, Volume 30.
Walter de Gruyter \& Co., Berlin (2011), xii+339.

\bibitem[Gab1]{Gab1}
{\sc P. Gabriel},
`{Auslander-{R}eiten sequences and representation-finite algebras}',
{\it Representation theory, {I} ({P}roc. {W}orkshop, {C}arleton {U}niv., {O}ttawa, {O}nt., 1979)},
1--71, {Lecture Notes in Math.}, {831}, Springer, Berlin, 1980.

\bibitem[Gab2]{Gab2}
{\sc P. Gabriel},
`{Unzerlegbare Darstellungen I}',
{\it Manuscripta Math.} 6 (1972), 71--103.

\bibitem[GeiLen]{GeiLen}
{\sc W. Geigle, H. Lenzing.}
`{A class of weighted projective curves arising in representation theory of finite-
dimensional algebras}', 
{\it Singularities, representation of algebras, and vector bundles (Lambrecht, 1985)},
265–-297, Lecture Notes in Math., 1273, Springer, Berlin, 1987.

\bibitem[GLS]{GLS}
{\sc C. Geiss, B. Leclerc, J. Schr\"{o}er},
`{Quivers with relations for symmetrizable Cartan matrices I : Foundations}',
{\it Invent. Math.} (2017), doi:10.1007/s00222-016-0705-1

\bibitem[Hap1]{Hap1}
{\sc D. Happel},
`{On Gorenstein algebras}',
{\it Representation theory of finite groups and finite-dimensional algebras (Bielefeld, 1991)}, 389--404, Progr. Math., 95, Birkh\"auser, Basel, 1991.

\bibitem[HerIya]{HerIya}
{\sc M. Herschend, O. Iyama},
`{{$n$}-representation-finite algebras and twisted fractionally Calabi-Yau algebras}',
{\it Bull. London Math. Soc.} (2011) 43, no. 3, 449--466.

\bibitem[HIMO]{HIMO}
{\sc M. Herschend, O. Iyama, H. Minamoto, S. Oppermann}, 
`{Representation theory of Geigle-Lenzing complete intersections}'. arXiv:1409.0668

\bibitem[HIO]{HIO}
{\sc M. Herschend, O. Iyama, S. Oppermann}, 
`{{$n$}-representation infinite algebras}',
{\it Adv. Math.} 252 (2014), 292--342.

\bibitem[Iya1]{Iya1}
{\sc O. Iyama},
`Auslander correspondence',
{\it Adv. Math.} 210 (2007), no. 1, 51--82.

\bibitem[Iya2]{Iya2}
{\sc O. Iyama},
`Cluster tilting for higher Auslander algebras',
{\it Adv. Math.} 226 (2011), no. 1, 1--61.

\bibitem[IS]{IS}
{\sc O. Iyama, {\O}. Solberg},
`Auslander-Gorenstein algebras and precluster tilting'.  arXiv: 1608.04179

\bibitem[Kue]{Kue}
{\sc J. K\"{u}lshammer},
`Pro-Species of Algebras I: Basic Properties',
to appear in {\it Algebras and Representation Theory} (2017), 1--24. doi:10.1007/s10468-017-9683-2

\bibitem[LvZh]{LvZh} 
{\sc H. Lv, S. Zhang},
`Global dimensions of endomorphism algebras of generator-cogenerators over $m$-replicated algebras',
{\it Comm. Algebra} 39 (2011), no. 2, 560--571. 

\bibitem[Mar]{Mar} 
{\sc R. Marczinzik}, 
`Upper bounds for the dominant dimension of Nakayama and related algebras', preprint. arXiv: 1605.09634

\bibitem[MiyYek]{MiyYek}
{\sc J. Miyachi, A. Yekutieli},
`{Derived Picard groups of finite-dimensional hereditary algebras}',
{\it Compositio Math.} 129 (2001), no. 3, 341--368.

\bibitem[PS]{PS}
{\sc J.A. de la Pena, A. Skowronski},
`The Tits forms of tame algebras and their roots',
{\it Representatios of algebras and related topics}, 
445--499, EMS Ser. Congr. Rep., 2011. 

\bibitem[QPA]{QPA} 
{\sc The QPA-team.} 
QPA - Quivers, path algebras and representations - a GAP package, Version 1.25; 2016
\newline\noindent (\url{https://folk.ntnu.no/oyvinso/QPA/})

\bibitem[Rin]{Rin}
{\sc C.M. Ringel.}
{\it Tame algebras and integral quadratic form.}
{Lecture Notes in Math.} 1099, Springer-Verlag, Berlin, 1984.

\bibitem[SkoYam]{SkoYam}
{\sc A. Skowronski, K. Yamagata},
{\it Frobenius Algebras I: Basic Representation Theory.}
EMS Textbooks in Mathematics, 2011.

\bibitem[Ta]{Ta}
{\sc H. Tachikawa},
{\it Quasi-Frobenius Rings and Generalizations: QF-3 and QF-1 Rings.}
Lecture Notes in Mathematics, Vol. 351, Springer-Verlag, Berlin-New York, 1973.

\bibitem[Thr]{Thr}
{\sc R. M. Thrall},
`Some generalizations of quasi-Frobenius algebras',
{\it Trans. Amer. Math. Soc.} 64 (1948), 173--183.

\bibitem[Vas]{Vas}
{\sc L. Vaso},
`$n$-cluster tilting subcategories of representation-directed algebras', preprint. arXiv:1705.01031

\bibitem[Yam]{Yam} 
{\sc K. Yamagata},
`Frobenius algebras',
{\it Handbook of algebra}, Vol. 1, 841--887, North-Holland, Amsterdam, 1996. 

\end{thebibliography}
\end{document}